\documentclass{amsart}
\DeclareFontFamily{OML}{script}{}
\DeclareFontShape{OML}{script}{m}{it}
{ <5-20> rsfs10 }{}
\DeclareMathAlphabet{\mathscript}{OML}{script}{m}{it}

\renewcommand{\mathcal}[1]{{\mathscript #1}\hspace{0.2ex}}
\usepackage{color}
\definecolor{xing}{RGB}{138, 43, 226}
\ifx\red\undefined
\newcommand{\red}{\color{red}}

\fi

\usepackage{graphicx,graphics}
\usepackage{cite}
\usepackage{pifont}
\usepackage{amsthm}
\usepackage{CJKutf8}
\AtBeginDocument{\begin{CJK}{UTF8}{gbsn}}
\AtEndDocument{\end{CJK}}

\usepackage{amscd}
\usepackage{amsmath}
\usepackage{latexsym}
\usepackage{amsfonts}
\usepackage{amssymb}
\usepackage{color}
\usepackage{multicol}
\usepackage{bm}
\allowdisplaybreaks[4]

\newcommand{\re}[1]{\mbox{\rm$($\ref{#1}$)$}}
\newcommand{\dis}{\displaystyle}
\newcommand{\m}{\hspace{1em}}
\newcommand{\mm}{\hspace{2em}}

\newcommand{\xx}{\vspace*{2ex}}

\newcommand{\Rmnum}[1]{\uppercase\expandafter{\romannumeral #1}}

\renewcommand{\epsilon}{\varepsilon}

\ifx\text\undefined
\newcommand{\text}{\mbox}
\fi
\ifx\operatorname\undefined
\newcommand{\operatorname}{\mathop}
\fi

\newcommand\be{\begin{equation}}
\newcommand\ee{\end{equation}}
\newcommand\bea{\begin{eqnarray}}
\newcommand\eea{\end{eqnarray}}
\newcommand\beaa{\begin{eqnarray*}}
\newcommand\eeaa{\end{eqnarray*}}

\newcommand{\spn}{\mathrm{span}}
\newcommand{\codim}{\mathrm{codim}}
\newcommand{\im}{\mathrm{Im}}

\newenvironment{eqa}{\begin{equation}%
  \begin{array}{rcl}}{\end{array}\end{equation}}

\newcommand\beqa{\begin{eqa}}
\newcommand\eeqa{\end{eqa}}

\numberwithin{equation}{section}

\renewcommand{\tilde}{\widetilde}
\renewcommand{\hat}{\widehat}
\renewcommand{\bar}{\overline}

\newtheorem{thm}{Theorem}[section]
\newtheorem{cor}[thm]{Corollary}
\newtheorem{lem}[thm]{Lemma}
\newtheorem{rem}[thm]{Remark}

\newcommand{\void}[1]{}

\newcommand{\bR}{{\mathbb R}}  

\numberwithin{equation}{section}

\begin{document}
\title[
periodic solutions for a free-boundary tumor model]{Symmetry-breaking bifurcation of periodic solutions for a free-boundary tumor model\footnote{\today}
}

\footnote{* Corresponding author:~Ruixiang Xing ~~~~~~ $^{*}$E-mail\,$: xingrx@mail.sysu.edu.cn $}
\footnote{$^{1}$E-mail\,$:hwh@sxu.edu.cn $  $^2$E-mail\,$:mxwang@sxu.edu.cn $}
 \author[Wenhua He, Mingxin Wang, and Ruixiang Xing]{ Wenhua He$^\dag$, Mingxin Wang$^\dag$, and Ruixiang Xing$^{*}$ \\ $^\dag$School of Mathematics and Statistics, Shanxi University, Taiyuan 030006, China\\
 $^{*}$School of Mathematics, Sun Yat-sen University, Guangzhou 510275, China}

\begin{abstract}
In this paper, we consider a free boundary multi-layer tumor model that incorporates a $T-$periodic provision of external nutrients $\Phi(t)$.
The simplified model contains three parameters: the mean of periodic external nutrients $\Phi(t)$, the threshold concentration $\widetilde{\sigma}$ for proliferation and the cell to cell adhesiveness coefficient $\gamma$.
We first study the flat solution and give a complete classification about $\frac{1}{T} \int_0^T \Phi(t) d t$ and $\widetilde{\sigma}$ according to global stability of zero equilibrium solution or global stability of the positive periodic solution.
Precisely,
(i) a zero flat solution is globally stable under the flat perturbations if and only if $\widetilde{\sigma} \geqslant \frac{1}{T} \int_0^T \Phi(t) d t$; (ii) If $\widetilde{\sigma}<\frac{1}{T} \int_0^T \Phi(t) d t$, then there exists a unique positive flat solution $\left(\sigma_*(y, t), p_*(y, t), { \rho_*(t)}\right)$ with period $T$ and it is a global attractor of all positive flat solutions for all $\gamma>0$. We further investigate periodic solutions bifurcating from the flat periodic solution $\left(\sigma_*(y, t), p_*(y, t), { \rho_*(t)}\right)$. By periodicity and symmetry,
we not only give symmetry-breaking periodic solutions for all positive parameter $\gamma_j$, but also show the existence of a plethora of periodic bifurcations. For the free boundary tumor problem, this is the first result of the existence of periodic bifurcations.

 \xx\noindent
{\bf Keywords.}
Tumor growth; Free boundary problem; Periodic solution; Bifurcation.

\xx\noindent
{\bf 2010 mathematics subject classifications.} 34C25, 35R35, 35R37,  35Q92, 92C37
\end{abstract}

\maketitle
\section{Introduction}
We consider a free boundary problem modeling a multi-layer tumor with a periodic provision of external nutrients. Let
$$
\Omega(t) = \{(x_1,x_2,y)\in \bR^{2}\times \bR \;\big|\;  0<y< \rho(x_1,x_2,t)\} 
$$
be a flat-shaped region of a 3-dimensional tumor where the positive function $\rho$ is unknown.
Denote $\Gamma(t)$ by  the upper free boundary
$\{y | y= \rho(x_1,x_2,t)\}$ of $\Omega(t)$, which is a permeable layer and  $\Gamma_0$ by the lower boundary $\{y|y=0\}$, which is fixed and impermeable.
The flat-shaped domain of tumor is used to study the metabolism of multi-layer tumor tissues.
For more details, we refer to the papers \cite{1997Three,2004Three,kyle1999characterization}.

The concentration $\sigma$ of nutrient (e.g., oxygen or glucose) satisfies the reaction-diffusion equation:
\begin{eqnarray}\label{1.1}
\lambda\sigma_t - \Delta\sigma + \sigma = 0, \hspace{2em} (x_{1}, x_{2}, y)\in\Omega(t),\hspace{0.5em} t>0,
\end{eqnarray}
with the boundary condition
\begin{eqnarray}
\displaystyle\frac{\partial \sigma}{\partial y}\Big|_{\Gamma_0} =0 , \hspace{2em} \sigma \Big|_{\Gamma(t)} = \Phi(t), \hspace{2em} t>0.\label{1.2}
\end{eqnarray}
We assume that external nutrient concentration $\Phi(t)$ is a positive continuous function with a period $T$ instead of constant external nutrient concentration \cite{CE1}. It is more reasonable. Here, $\lambda$ denotes the ratio of the {nutrients'} diffusion rate
to the
cell proliferation rate
and $\lambda\ll1$  (see \cite{Byrne1995}). In this paper, we consider the quasi-steady state approximation, i.e.,
$\lambda=0$.

Let $p$ be the pressure, $S$ represent the proliferation rate and $\vec{V} $ denote the velocity of the tumor cell movement.
 Assuming that the tumor is porous medium type, Darcy's law ($\vec{V}=-\nabla p$) and the conversation of mass ($\mbox{div} \vec V = S$) imply $ -\Delta p = S$.  Suppose that $S=  \mu(\sigma-\tilde{\sigma})$, where $\mu$
is the tumor aggressiveness constant and $\tilde \sigma$
represents the threshold concentration for proliferation.
Then $p$ satisfies the following equation:
\begin{equation}\label{1.5b}
-\Delta p =\mu(\sigma-\tilde{\sigma}),\hspace{2em} (x_{1}, x_{2}, y)\in\Omega(t),\hspace{0.5em} t>0,
\end{equation}
with the boundary condition
\begin{eqnarray}\label{1.6}
 \displaystyle\frac{\partial p}{\partial y} \Big|_{\Gamma_0}=0, \hspace{2em} p \Big|_{\Gamma(t)}= \gamma\kappa ,\hspace{2em} t>0,
\end{eqnarray}
 where $\gamma$ is the cell to cell adhesiveness and $\kappa$ is the mean curvature.
Supposing that the velocity field is continuous to the free boundary, then the normal velocity on
  $\Gamma(t)$ is
\begin{equation}\label{1.8}
V_n =  -\frac{\partial p}{\partial n},  \hspace{2em} (x_{1}, x_{2},y)\in \Gamma(t),\hspace{0.5em} t>0,
\end{equation}
where $n$ is the unit outside the normal vector.
The initial value of domain $\Omega(t)  $ is
$
 \Omega_0.
$

There are some interesting results for the multi-layer tumor model with \textit{constant} external nutrient concentration. For the quasi-steady state approximation,
  Cui and Escher \cite{CE1} have established local well-posedness by means of the analytic
 semigroup theory. Also under the assumption $\Phi>\widetilde{\sigma}$, they have shown the existence and uniqueness of the positive flat stationary solution and its asymptotic behavior under non-flat perturbations.
For the 2-dimensional tumor model, Zhou, Escher and Cui \cite{2008Bifurcation} have studied the stationary bifurcation of the corresponding steady-state problem.
In the presence of inhibitors, Zhou, Wu and Cui \cite{zwc} have derived the local existence and asymptotic behavior of flat stationary solutions under non-flat perturbations. For the 2-dimensional tumor problem, Lu and Hu \cite{2020Bifurcation} studied the stationary bifurcation of tumor growth with ECM
and MDE interactions.
Recently, for a tumor model with time delay, He, Xing and Hu considered the linear stability of the  flat stationary solution under non-flat perturbations for quasi-steady state approximation in \cite{hxh1} and for general case with $\lambda\neq 0$ in \cite{hxh2}, respectively.
He and Xing \cite{hx2023} have established the existence of the stationary bifurcations for a tumor model with time delay.

For the classical tumor growth models with a sphere-shaped domain and periodic external nutrients, Bai and Xu \cite{Bai2013} have shown that the zero equilibrium solution is globally stable if  $\widetilde{\sigma}>\frac{1}{T} \int_{0}^{T} \Phi(t) d t$ and if the zero equilibrium solution is globally stable, $\widetilde{\sigma}\ge\frac{1}{T} \int_{0}^{T} \Phi(t) d t$. Also, they have proved the existence, uniqueness and stability of the positive periodic solution under the assumption $\min _{0 \le t \le T} \Phi(t)>\widetilde{\sigma} $. When $\Phi(t)$ is not a constant function, then $\min _{0 \le t \le T} \Phi(t)<\frac{1}{T} \int_{0}^{T} \Phi(t) d t$ and there is a gap $\widetilde{\sigma}\in[\min _{0 \le t \le T} \Phi(t), \frac{1}{T} \int_{0}^{T} \Phi(t) d t)$ to be answered.
{For the 2-dimensional  quasi-steady state problem}, Huang, Zhang and Hu \cite{Huang2019} and Huang \cite{huang2022} have described the linear stability and asymptotic stability of the positive $T-$periodic solution under non-radial perturbations. Recently, He and Xing \cite{he} have filled the gap in \cite{Bai2013}, given a complete classification about $ \Phi(t)$ and $\widetilde{\sigma}$ according to the global stability of zero equilibrium and the existence of periodic solutions, and shown the linear stability of the positive $T-$periodic solution under non-radial perturbations for the 3-dimensional case.

{In this paper,} we first study the flat solution and give a complete classification about $\frac{1}{T} \int_0^T \Phi(t) d t$ and $\widetilde{\sigma}$ according to the global stability of zero equilibrium solution or global stability of the positive periodic solution.
The detailed results are stated in Theorems \ref{thm:1.1a'}--\ref{thm:1.2}. The proof of Theorems \ref{thm:1.1a'}--\ref{thm:1.2} follows the ideas in \cite{he}.
Unlike the sphere-shaped model in \cite{he}, our model has a flat-shaped domain that causes various distinct computations and estimates. These are not the main results of this paper.
For the sake of the proof's completeness and periodic solutions used in the below branching results, we give the proof in the appendix.

The corresponding flat problem of \re{1.1}--\re{1.8} is
\begin{align}\label{1.10a}
&\frac{\partial^{2} \sigma}{\partial y^{2}}=\sigma    && y \in(0, \rho(t)),\hspace{0.5em}  t>0,\\\label{1.11a}
&\displaystyle\frac{\partial \sigma}{\partial y} (0,t)=0,  \mm   \sigma(\rho(t),t)=\Phi(t), &&     t>0,\\\label{1.13a}
&-\frac{\partial^{2} p}{\partial y^{2}}=\mu(\sigma-\widetilde{\sigma})  && y \in(0, \rho(t)),  \hspace{0.5em}t>0,\\\label{1.14a}
&\displaystyle\frac{\partial p}{\partial y} (0,t)=0, \mm p(\rho(t),t)= 0,  &&    t>0,\\\label{1.15a}
&\frac{d \rho}{d t}=-\frac{\partial p}{\partial y}  && y=\rho(t), \hspace{0.5em} t>0,\\\label{1.17a}
&\rho(0)=\rho_{0}.
\end{align}
The solution of \re{1.10a}--\re{1.14a} satisfies
\begin{align}\label{1.13aa}
&\sigma(y, t)=\Phi(t)\frac{\cosh y}{\cosh(\rho(t))},\\
\label{1.14aa}
&p(y,t)=  \frac{1}{2}\mu \tilde{\sigma} y^2 +\mu\Phi(t)- \frac{1}{2}\mu\tilde{\sigma}\rho^2(t)-\mu \Phi(t)\frac{\cosh y}{\cosh (\rho(t))}.
\end{align}
From \re{1.13aa} and \re{1.14aa}, system \re{1.10a}--\re{1.17a} is reduced to the following system:
\begin{align}\label{1.15}
&\frac{d \rho}{d t}(t)=\mu \rho(t)\left[\Phi(t) \frac{ \tanh( \rho(t))}{\rho(t)}-\widetilde{\sigma}\right],\\ \label{1.155}
&\rho(0)=\rho_{0}.
\end{align}
 Notice that $\rho$ is a solution of \re{1.15}--\re{1.155} if and only if  $(\sigma, p, \rho)$  is a solution of \re{1.10a}--\re{1.17a}, where $\sigma$ and $ p$ are given in \re{1.13aa} and \re{1.14aa}, respectively.

Denote
\begin{align}\nonumber
\overline{\Phi}=\frac{1}{T} \int_{0}^{T} \Phi(t) d t.
\end{align}

Our results about flat solutions are the following theorems.
\begin{thm}\label{thm:1.1a'}
For any initial value $\rho_{0}>0$, \re{1.15}--\re{1.155} has a unique positive global solution $\rho$.
\end{thm}

\begin{thm}\label{thm:1.1'}
 $\widetilde{\sigma}\ge\overline{\Phi}$ if and only if the zero solution of system \re{1.15}--\re{1.155} is globally stable.
\end{thm}

\begin{thm}\label{thm:1.2}
If $\widetilde{\sigma}<\overline{\Phi}$, then the following conclusions hold: \\
(i) \re{1.15}--\re{1.155} has a unique positive $T-$periodic  solution $\rho_{\ast}$. \\
(ii)  There exist $\delta>0$ and $C>0$ such that for any the positive solution $\rho$,
\begin{align}
|\rho(t)-\rho_{\ast}(t)|\le C e^{-\delta t}\qquad for ~    t>0. \ \label{1.17}
\end{align}
\end{thm}

Theorem \ref{thm:1.1'} demonstrates that when the mean value of external nutrients is inadequate for tumor cell proliferation, all flat-shaped tumors will vanish. On the other hand, Theorem \ref{thm:1.2} shows that if the external nutrient is sufficient, all flat-shaped tumors grow into a $T-$periodic state. These two theorems give an insight into the relationship between external nutrients and tumor growth behavior.

From Theorem \ref{thm:1.2}, we get the following corollary.
\begin{cor}\label{cor:1.3333}
 If $\widetilde{\sigma}<\overline{\Phi}$, then there exists a unique positive $T-$periodic solution $(\sigma_{*}, p_{*}, \rho_{*})$ of \re{1.10a}--\re{1.17a}, where $\rho_{*}$ is the unique positive $T-$periodic solution of \re{1.15}--\re{1.155} and $\sigma_{*}$ and $p_{*}$ are given by \eqref{1.13aa} and \eqref{1.14aa} with $\rho=\rho_{*}$, respectively.
\end{cor}

\void{\begin{rem} The theorems \ref{thm:1.1a'} $-$ \ref{thm:1.2} are the basis for the study of the branching solutions below and are not the main results. In this article, we will focus on branching solutions, but for the completeness of the proof, we will include the proof of the theorem in the appendix.
\end{rem}}

We shall derive the non-flat periodic bifurcations from the flat periodic solution $\rho_{*}$.
Regarding symmetry-breaking bifurcation for free boundary problems,
all the obtained results so far describe the existence of stationary branches stemming from the radially symmetric stationary solutions or flat stationary solutions (see \cite{Wang2017,Escher2011,WZ2,Wu2015,Cui2007,CE2,Zhang2009,Hao1,Li2017,Friedman2007a,Pan2018a,Wang2014,Huang2017,crandall,Friedman2006,  w, Cui2018,Friedman2007b, 2008Bifurcation}).
 In 2008, Friedman and  Hu \cite{MR2415075} investigated
the classical Byrne-Chaplain tumor model and got
the existence of axisymmetric stationary branches bifurcating from the radially symmetric stationary solutions for $\mu_n$ with even $n$.
Recently, Pan and Xing \cite{MR4524874}
gave
many stationary branches of non-radially symmetric solutions
at $\mu=\mu_n (n=2,3, \cdots)$.

For the \textit{linearized problem} of the classical Byrne-Chaplain free boundary tumor model, Friedman and  Hu \cite{MR2415075} established the existence of periodic solutions  through extensive and intricate calculations.
{In this paper,} we study the quasi-steady state problem of \eqref{1.1}--\eqref{1.8} and derive periodic solutions bifurcating from the flat periodic solutions using Crandall-Rabinowitz theorem.
This is the first result about the existence of the branches of periodic solutions for a free boundary tumor model.

 Denote
\begin{equation}\label{1}
  \gamma_{j}(\rho_*)=\frac{ \mu k_1({j,\rho_*})}{  \; k_2({j,\rho_*})}, \mm j>j_0,
\end{equation}
 where
\begin{align}\label{2}
 k_1(j) & = \int_{0}^{T}\Phi(t)\Bigg\{1-\frac{\tanh (\rho_*(t))}{\rho_*(t)}
 -\tanh (\rho_*(t))\cdot \Big[\sqrt{1+j}\tanh(\sqrt{1+j}\rho_*(t) )\\\nonumber
 &\mm-
 \sqrt{j}\tanh(\sqrt{j}\rho_*(t))\Big]\Bigg\}d t,\\\label{3}
  k_2(j) & = \int_{0}^{T} \frac12  j^{3/2} \tanh(\sqrt{j}\rho_*(t)) d t.
\end{align}
It is also derived in \cite[section 4]{CE1},
\begin{equation}\nonumber
    \frac{\partial^{2} }{\partial x^{2}}[\sqrt{x}\tanh (\sqrt{x}\rho)]
=\dis\frac{\rho}{2}\frac{\partial}{\partial x}\;\frac{\sinh(\sqrt{x}\rho)\cosh(\sqrt{x}\rho)+\sqrt{x}\rho}{(\rho\sqrt{x})\cosh^{2}(\sqrt{x}\rho)}<0,
\end{equation}
for $x>0$ and $\rho>0$. Then $k_1(j)$ is strictly increasing in $j$. Since
\begin{eqnarray*}
&&
\frac{d}{d \rho}\frac{\tanh \rho}{\rho}=\frac{1}{ \rho}\Big(1-\frac{\tanh \rho}{\rho}-\tanh^{2}\rho\Big)=\displaystyle\frac{\rho-\sinh \rho\cosh \rho}{\rho^2\cosh^{2} \rho}<0,\hspace{2em} \rho>0,
\end{eqnarray*}
it follows
$$
k_1(0)\!=\!\!\int_{0}^{T}\!\!\Phi(t)\Big[1- \frac{\tanh (\rho_*(t))}{\rho_*(t)}- \tanh^{2}(\rho_*(t))\Big]d t<0,
\lim_{j\rightarrow+\infty} \!\!\!k_1(j)\!=\!\!\int_{0}^{T}\!\!\Phi(t)\Big[1-\frac{\tanh (\rho_*(t))}{\rho_*(t)}\Big]d t>0.
$$
Hence there exists a unique $j_{0} >0$ (not necessarily an integer) such that
\bea\nonumber
k_1(j)<0 \m\text{ for }  0\le j < j_{0}, \mm k_1(j_0)=0, \mm
k_1(j)>0 \m\text{ for }   j > j_{0}.
\eea
Hence, $ \gamma_{j}$ is positive for $j > j_{0}$.

Our main result is the following theorem.
\begin{thm}\label{result1}(Symmetry-Breaking)
Assume that $ \Phi\in C^{1+\frac{\alpha}{3}}{([0,\infty))}$ is a positive T--periodic function and  $j$ ($j > j_0$) is a sum of the squares of two non-negative
 integers.
 Then $\gamma=\gamma_{j}$ is a bifurcation value of system \eqref{1.10a}--\eqref{1.17a}. More precisely,
there exist at least $\beta$ bifurcation branches of non-flat periodic
solutions bifurcating from $\gamma=\gamma_{j}$  with free boundary
$$y=\rho_*(t) + \epsilon S_{n^2+m^2,\gamma_{j}}(t)\cos\left(n x_1\right)\cos\left(m x_2\right) + O(\epsilon^2).$$
If there is other $i$ such that $\gamma_i = \gamma_j$, then $\beta=\alpha_{\max\{i,j\}}$ and $ \max\{i,j\}=n^2+m^2$;
otherwise, $\beta=\alpha_j$ and  $ j=n^2+m^2$, where $
     \alpha_j=\text{ the number of } \{(n,m)\,:\,  j=n^2+m^2, n\geq m\}
$.
Here
$S_{j,\gamma_{j}}(t)=e^{-\int_{0}^{t} A_{j,\gamma_{j}}(s)ds}$, where
\begin{align}\label{Aj}
A_{j,\gamma}(t)&=\mu\widetilde{\sigma}-\mu \Phi(t)  +\mu\Phi(t)\tanh(\rho_*(t))\sqrt{1+j}\tanh(\sqrt{1+j}\rho_*(t))\\\nonumber
&\mm-\mu \widetilde{\sigma}\rho_*(t)\sqrt{j} \tanh (\sqrt{j}\rho_*(t))+\frac{\gamma j^{\frac{3}{2}}}{2}\tanh (\sqrt{j}\rho_*(t)).
\end{align}

If $j=n^2$, $\min\{i,j\}=n^2$ or $\max\{i,j\}=n^2$, there is an additional branch with
the free boundary
\begin{align*}
y=\rho_* + \epsilon S_{n^2,\gamma_{n^2}}(t)[\cos(nx_1)+\cos(nx_2)] + O(\epsilon^2).
\end{align*}
where $n=\sqrt{\max\{i,j\}}$ if both $\max\{i,j\}$ and $\min\{i,j\}$ are perfect square numbers.
\end{thm}

\begin{rem}
Note that $S_{j,\gamma_{j}}(t)\cos(mx_1)\cos(nx_2)$ is a $\frac{\pi}{2}$ rotation of $S_{j,\gamma_{j}}(t)\cos(nx_1)\cos(mx_2)$ in the $(x_1,x_2)$ plane.
Also, $S_{j,\gamma_{j}}(t)[\cos(nx_1)-\cos(nx_2)]$
is a translation and symmetry of $S_{j,\gamma_{j}}(t)
   [\cos($\\ $  nx_1)+\cos(nx_2)]$ in the $x_2$-direction.
Rotated, translated and symmetric solutions are not new solutions, so they do not contribute to the number of branches.
\end{rem}

\begin{rem}
Under the orthogonal transformation, $S_{j,\gamma_{j}}(t)[\cos(nx_1)+\cos(nx_2)]=2S_{j,\gamma_{j}}(t)\cos($\\$\frac{n}{\sqrt{2}}y_1)\cos(\frac{n}{\sqrt{2}}y_2)$.
Since $j= \frac{n^2}{2}+ \frac{n^2}{2} $ is not a sum of squares of two integers,
$S_{j,\gamma_{j}}(t)[\cos(nx_1)+\cos(nx_2)]$ is a new and supplementary branch of $j=n^2$.
\end{rem}

In Lemmas \ref{lem6.2} and \ref{lem5.3}, it is shown that $\gamma_j$ either strictly {decreases} in $j$ or exhibits a non-monotonic pattern of first {increasing}  and then {decreasing}. As a consequence, it is possible that some $\gamma_j$  are not distinct from each other. Additionally, for a given $j$, there are multiple pairs of sum-of-squares decomposition. It is worth observing that if $(n,m)$ is a square decomposition of $j$, then $(m,n)$ is also one of $j$.
These observations imply that the dimension of the kernel space of the linearized operator may not be one, which { induces} complexity to the analysis of periodic bifurcations.
We follow the idea of \cite{hx2023} and leverage the properties of periodicity and symmetry to overcome these difficulties and give periodic bifurcations for all positive bifurcation values.

 The paper is organized as follows. We give
 the results of bifurcation in Section 2.
 In Appendix A.1, we study flat solutions and give the necessary and sufficient conditions for the global stability of zero equilibrium solution. We show the existence, uniqueness and stability of the positive flat periodic solution in Appendix A.2.

\section{The results of bifurcation}

In this section, we apply Crandall-Rabinowitz theorem
to establish the existence of the $T-$periodic bifurcation from the positive $T-$periodic solution $\rho_*$.

Now we recall Crandall-Rabinowitz theorem.
\begin{thm}{\bf (Crandall-Rabinowitz theorem \cite{crandall})}\label{thmcr}
Let $X$, $Y$ be real Banach spaces and  $G(\cdot,\cdot)$ be a $C^p$ $(p\ge 2)$ mapping of a neighborhood $(0,\gamma_0)$ in $X \times \mathbb{R}$ into $Y$. Suppose
\begin{itemize}
\item[(1)] $G(0,\gamma) = 0$ for all $\gamma$ in a neighborhood of $\gamma_0$,
\item[(2)] $\mathrm{Ker} \,G_x(0,\gamma_0)$ is one dimensional space, spanned by $x_0$,
\item[(3)] $\mathrm{Im} \,G_x(0,\gamma_0)=Y_1$ has codimension 1,
\item[(4)] $G_{\gamma x}(0,\gamma_0) x_0 \notin Y_1$.
\end{itemize}
Then $(0,\gamma_0)$ is a bifurcation point of the equation $G(x,\gamma)=0$ in the following sense: In a neighborhood of $(0,\gamma_0)$ the set of solutions $G(x,\gamma) =0$ consists of two  $C^{p-1}$ smooth curves $\Gamma_1$ and $\Gamma_2$ which intersect only at the point $(0,\gamma_0)$; $\Gamma_1$ is the curve $(0,\gamma)$ and $\Gamma_2$ can be parameterized as follows:
$$\Gamma_2: (x(\epsilon),\gamma(\epsilon)), |\epsilon| \text{ small, } (x(0),\gamma(0))=(0,\gamma_0),\; x'(0)=x_0.$$
\end{thm}

To establish the smoothness of the corresponding non-linear mapping $ G(\tilde{\rho}, \gamma)$  and demonstrate that its derivative $  G_{\tilde{\rho}}(0,\gamma)$ is a Fredholm operator with index zero, we reformulate the free boundary problem as a fixed boundary problem. This reformulation offers significant advantages for theoretical analysis. However, directly computing the derivative of the nonlinear mapping from the free boundary is comparatively simpler. Hence, we establish three equivalent problems in Section 2.1. In Section 2.2, we show that the problem maintains periodicity and symmetry. In Section 2.3, we provide the results of bifurcation.

\subsection{Establishing Equivalent Problems
}\label{2.1aaaa}

Following ideas of \cite{2009Well},
through the Hanzawa transformation,  we transform \eqref{1.1}--\eqref{1.8} with $\lambda=0$ into a problem in the fixed domain $\Omega:=(0, 2\pi)\times  (0, 2\pi) \times(0,1)$.

Now, we recall the results in \cite[\S 2]{2009Well}. In \cite{2009Well}, the problem was studied in the case independent of \(x_2\) (i.e., the two-dimensional tumor problem). Their results can be extended similarly to the case of \(n = 3\).

The little H\"{o}lder spaces $h^{k+\alpha}(\overline{\Omega})$ ($k\in\mathbb{N}$ and $\alpha\in(0,1)$) is the closure of $C^\infty(\overline{\Omega})$ in $C^{k+\alpha}(\overline{\Omega})$.
Denote
$D=(0, 2\pi)\times  (0, 2\pi) $.

We study $2\pi$-periodic functions in the spatial direction. Define
\begin{align*}
  &h^{k+\alpha}_{2\pi}\left(\bar{D}\right):=\{\hat{\rho}\in  h^{k+\alpha}(\mathbb{R}^2) ~|~ \hat{\rho}(x_1+2\pi,x_2 )=\hat{\rho}(x_1,x_2)   , ~~ \hat{\rho}(x_1,x_2+2\pi)=  \hat{\rho}(x_1,x_2) \}.
\end{align*}
Let $ h^{4+\alpha}_+\left({\bar{D}}\right):=\{\hat{\rho}\in h^{4+\alpha}_{2\pi}\left(\bar{D}\right) ~|~ \hat{\rho}(x_1,x_2) >0\}.$

Given $\hat{\rho} \in h^{4+\alpha}_+\left({\bar{D}}\right)$, define
$$
\theta_{\hat{\rho}}: \bar{\Omega} \rightarrow \bar{\Omega}_{\hat{\rho}}, \quad\left(x_1^{\prime}, x_2^{\prime}, y^{\prime}\right) \mapsto \left(x_1, x_2, y\right)=\left(x_1^{\prime}, x_2^{\prime}, y^{\prime}   \hat{\rho}\left(x_1^{\prime}, x_2^{\prime}\right) \right).
$$
Then
$$
\theta_{\hat{\rho}}^{-1}:   \bar{\Omega}_{\hat{\rho}} \rightarrow   \bar{\Omega}, \quad  \left(x_1, x_2, y\right)  \mapsto \left(x_1^{\prime}, x_2^{\prime}, y^{\prime}\right) =\left(x_1, x_2, \frac{y}{   \hat{\rho}\left(x_1, x_2 \right) } \right).
$$

This diffeomorphism $\theta_{\hat{\rho}}$ induces the following pull-back and push-forward operators:
$$
\theta_{\hat{\rho}}^* u:=u \circ \theta_{\hat{\rho}},\quad u \in C\left(\bar{\Omega}_{\hat{\rho}}\right) \quad \text { and } \quad \theta_*^{\hat{\rho}} v:=v \circ \theta_{\hat{\rho}}^{-1}, \quad v \in C(\bar{\Omega}).
$$
 Let
$$
  {A}(\hat{\rho}) v:=-\theta_{\hat{\rho}}^* \Delta\left(\theta_*^{\hat{\rho}} v\right), \quad   {B}(\hat{\rho}) v:=\theta_{\hat{\rho}}^*\left(\operatorname{tr} \nabla\left(\theta_*^{\hat{\rho}} v\right) , \bm{v}_{\hat{\rho}}\right), \qquad \text{for  } v \in C^2(\bar{\Omega}),
$$
where $\operatorname{tr}$ is the trace operator with respect to $\Gamma_{\hat{\rho}}$ and $\bm{v}_{\hat{\rho}}=\left(-\nabla\hat{\rho}, 1\right)$ is the outer normal of $\Gamma_{\hat{\rho}}$.
Denote
\begin{align}\label{t1}
u(t):=\theta_{\rho(t)}^* \sigma( \cdot,t) \quad \text { and }  \quad v(t):=\theta_{\rho(t)}^* p( \cdot,t).
\end{align}
Since
$$
V_n = \frac{\rho_t}{\sqrt{1 + |\nabla \rho|^2}} \text{ and }  n= \frac{(-\nabla \rho, 1)}{\sqrt{1 + |\nabla \rho|^2}},$$
then  \eqref{1.1}--\eqref{1.8} with $\lambda=0$
is equivalent to the following problem in the fixed domain $\Omega={D} \times(0,1)$:
\begin{equation}\label{fq2.1}
\left\{\begin{aligned}
  &{A}(\rho(t)) u  + u=0 & & \text { in }(0, T) \times \Omega, \\
  &\frac{\partial u}{\partial y}  =0 & & \text { on }(0, T) \times D\times \{0\}, \\
  &u =\Phi(t)& & \text { on }(0, T) \times D\times \{1\}, \\
&  {A}(\rho(t)) v  =\mu(u-\tilde{\sigma}) & & \text { in }(0, T) \times \Omega, \\
&\frac{\partial v}{\partial y}=0 & & \text { on }(0, T) \times D\times \{0\}, \\
&v  =\gamma \kappa & & \text { on }(0, T) \times D\times \{1\}, \\
 &\frac{\partial \rho}{\partial t} =-  {B}(\rho(t)) v & & \text { on }(0, T) \times D\times \{1\}.
\end{aligned}\right.
\end{equation}

Next we continue to reduce \eqref{fq2.1} into an equation with only the unknown function $\rho$.

Given $\hat{\rho} \in h^{4+\alpha}_+\left(\bar{D}\right)$, there exists a unique solution in  $h^{4+\alpha}_{2\pi}(\bar{\Omega})$ of
$$
\left\{\begin{aligned}
 &{A}(\hat{\rho}) u+u=0 && \text { in } \Omega, \\
&\frac{\partial u}{\partial y} (\cdot, 0)=0, \quad u(\cdot, 1)=1 && \text { on } \bar{D},
\end{aligned}\right.
$$
which is
\begin{align}\nonumber
      {R}(\hat{\rho}) 1.
\end{align}
Here, $h^{k+\alpha}_{2\pi}(\bar{\Omega})$ is composed of functions from $h^{k+\alpha}(\mathbb{R}^2\times [0,1])$ that have $2\pi-$periodicity in $x_1$ and $x_2$.

The elliptic regularity theory implies
\begin{align}\nonumber
    {R}(\cdot)  1 \in C^{\infty}\left(h^{4+\alpha}_+\left(\bar{D}\right), h^{i+\alpha}_{2\pi}(\bar{\Omega})\right) \quad \text{     for }  i=2,3,4.
\end{align}
Then the unique solution of
$$
\left\{\begin{aligned}
 &{A}(\hat{\rho}) u+u=0 && \text { in } \Omega, \\
&\frac{\partial u}{\partial y} (\cdot, 0)=0, \quad u(\cdot, 1)=\Phi(t) && \text { on } \bar{D},
\end{aligned}\right.
$$
is given by
\begin{align}\label{f2.1aa}
    u(t)= \Phi(t) {R}(\hat{\rho}) 1.
\end{align}

Denote by ${S}\left(\hat{\rho} \right)$ and $\mathcal{T}\left( \hat{\rho} \right)$ the solution operators of the following problems,
respectively:
\begin{align}\nonumber
    \left\{  {\begin{array}{ll} A\left( \hat{\rho} \right) u = f & \text{ in }{\Omega }, \\ \dfrac{\partial u}{\partial y} (\cdot, 0)=0, \quad u(\cdot, 1)=0& \text{ on }\bar{D} ; \end{array}\;
    \qquad
    \left\{  \begin{array}{ll} A\left( \hat{\rho} \right) u = 0 & \text{ in }{\Omega }, \\ \dfrac{\partial u}{\partial y} (\cdot, 0)=0, \quad u(\cdot, 1) = g & \text{ on }\bar{D}, \end{array}\right. }\right.
\end{align}
then
\begin{align} \nonumber
&   {S} \in C^{\infty}\left(h^{4+\alpha}_+\left(\bar{D}\right),   {L}\left(h^{1+\alpha}_{2\pi}(\bar{\Omega}), h^{3+\alpha}_{2\pi}(\bar{\Omega})\right)\right),  \\\nonumber
&   \mathcal{T} \in C^{\infty}\left(h^{4+\alpha}_+\left(\bar{D}\right),   {L}\left(h^{2+\alpha}_{2\pi}\left(\bar{D}\right), h^{2+\alpha}_{2\pi}(\bar{\Omega})\right)\right).
\end{align}

Given
$
\hat{\rho} \in h^{4+\alpha}_+\left(\bar{D}\right) \text { and }(f, g) \in  h^{1+\alpha}_{2\pi}(\bar{\Omega}) \times h^{2+\alpha}_{2\pi}\left(\bar{D}\right),
$
the solution of
$$
\left\{\begin{aligned}
&  {A}(\hat{\rho}) p=f && \text { in } \Omega, \\
&\frac{\partial p}{\partial y} (\cdot, 0)=0, \quad p(\cdot, 1)=g && \text { on } \bar{D},
\end{aligned}\right.
$$
is
\begin{align}\label{f2.3}
p=  {S}(\hat{\rho}) f+  \mathcal{T}(\hat{\rho}) g.
\end{align}

Finally, given $\hat{\rho} \in h^{4+\alpha}_+\left(\bar{D}\right)$, let $  {P}(\hat{\rho})$ be the linear operator

\[
{P}(\hat{\rho})v = -\frac{(1+\hat{\rho}_{x_{2}^{\prime}}^2)v_{x_{1}^{\prime}x_{1}^{\prime}}+(1+\hat{\rho}_{x_{1}^{\prime}}^2)v_{x_{2}^{\prime}x_{2}^{\prime}}-2\hat{\rho}_{x_{1}^{\prime}}\hat{\rho}_{x_{2}^{\prime}}v_{x_{1}^{\prime}x_{2}^{\prime}}}{2\big(1+\hat{\rho}_{x_{1}^{\prime}}^2+\hat{\rho}_{x_{2}^{\prime}}^2\big)^{3/2}}.
\]
Hence the curvature is given by $\kappa(\hat{\rho})=  {P}(\hat{\rho}) \hat{\rho}$ and
\begin{align}
\nonumber
  {P} \in C^{\infty}\left(h^{4+\alpha}_+\left(\bar{D}\right),   {L}\left(h^{4+\alpha}_{2\pi}\left(\bar{D}\right), h^{2+\alpha}_{2\pi}\left(\bar{D}\right)\right)\right).
\end{align}

Together with \re{f2.3} and \re{f2.1aa}, for
$
\hat{\rho} \in h^{4+\alpha}_+\left(\bar{D}\right)$, the unique solution of $\eqref{fq2.1}_{4}$--$\eqref{fq2.1}_{6}$ is
\begin{align*}
 v = \mu {S}\left( \hat{\rho}\right)\left( u- \tilde{\sigma}\right)  + \gamma  \mathcal{T}\left( \hat{\rho}\right) {P}(\hat{\rho}) \hat{\rho}  =
  \mu \Phi(t) {S}\left( \hat{\rho} \right) {R}(\hat{\rho}) 1
  - \mu {S}\left(\hat{\rho} \right) \tilde{\sigma}   + \gamma \mathcal{T}\left(\hat{\rho} \right) {P}(\hat{\rho}) \hat{\rho}  .
\end{align*}
Then
\begin{align*}
{B}(\hat{\rho}) v   = \gamma {B}(\hat{\rho}) \mathcal{T}\left(\hat{\rho} \right) {P}(\hat{\rho}) \hat{\rho}
  + \mu \Phi(t) {B}(\hat{\rho}) {S}\left( \hat{\rho} \right) {R}(\hat{\rho}) 1
  - \mu  {B}(\hat{\rho}) {S}\left(\hat{\rho} \right) \tilde{\sigma}  .
\end{align*}

Let
\begin{align}\label{Q2.2}
\Psi_1(\hat{\rho}):=  {B}(\hat{\rho})   \mathcal{T}(\hat{\rho})   {P}(\hat{\rho}), \quad \Psi_2(\hat{\rho}):=\mu   {B}(\hat{\rho}) {S}(\hat{\rho}){R}(\hat{\rho}) 1 , \quad  \Psi_3(\hat{\rho}):= \mu {B}(\hat{\rho}) {S}(\hat{\rho})\tilde{\sigma},
\end{align}
for $\hat{\rho} \in h^{4+\alpha}_+\left(
\bar{D}\right)$.
Then
\begin{align}\label{rq1}
(\Psi_1, \Psi_2, \Psi_3) \in C^{\infty}\left(h ^{4+\alpha}_+\left(\bar{D}\right),   {L}\left(h^{4+\alpha}_{2\pi}\left(\bar{D}\right), h^{1+\alpha}_{2\pi}\left(\bar{D}\right)\right) \times h^{2+\alpha}_{2\pi}\left(\bar{D}\right) \times h^{2+\alpha}_{2\pi}\left(\bar{D}\right) \right).
\end{align}
The Nemytschi operator \( F \) is defined as
\begin{align}
\label{f2.6f}
 {F}:\left( {{\rho}(t) ,\gamma }\right)  \mapsto  - \gamma\Psi_1 ({\rho}(t)){\rho(t)} - \Phi(t)  \Psi_2({\rho(t)}) + \Psi_3 ({\rho(t)}) .
\end{align}

Problem \eqref{fq2.1} is equivalent to
\begin{align}\label{Q222}
\frac{d \rho}{d t}  =F(\rho,\gamma).
\end{align}

To apply the method of Lyapunov-Schmidt and
get the $T-$periodic solution,
we introduce the
following Banach spaces $E$, $W$ and $Y$ of \( T- \)periodic H\"{o}lder continuous functions taking
values in \( X \) or \( Z \) ,
where
\begin{align*}
  &X:=\{\hat{\rho}\in  h^{4+\alpha}_{2\pi}\left(\bar{D}\right) ~|~ \hat{\rho}(-x_1,x_2)=\hat{\rho}(x_1,-x_2)=  \hat{\rho}(x_1,x_2)   \},\\
   & Z := \{\hat{\rho}\in  h^{1+\alpha}_{2\pi}\left(\bar{D}\right) ~|~ \hat{\rho}(-x_1,x_2)=\hat{\rho}(x_1,-x_2)=  \hat{\rho}(x_1,x_2)   \},\\
 &  E := {C}_{T }^{\frac{\alpha}{3} }\left( {\mathbb{R},X}\right)  =  \left\{  {x : \mathbb{R} \rightarrow  X \mid  x\left( {t + T}\right)  = x\left( t\right) ,t \in  \mathbb{R},}\right. \\
&\qquad\qquad \| x{\| }_{E} = \| x{\| }_{X,\frac{\alpha}{3} } =  \mathop{\max }\limits_{{t \in  \mathbb{R}}}\| x(t){\| }_{X} + \mathop{\sup }\limits_{{s \neq  t}}\frac{\| x\left( t\right)  - x( s){\|}_{X}}{{\left| t - s\right| }^{\frac{\alpha}{3} }} < \infty \} , \\
& W :=  {C}_{T }^{\frac{\alpha}{3} }\left( {\mathbb{R},Z}\right) \text{ analogously},\\ \nonumber
& Y :=  {C}_{T }^{1 + \frac{\alpha}{3} }\left( {\mathbb{R},Z}\right)  \equiv  \left\{  {x : \mathbb{R} \rightarrow  Z \mid  x,\frac{dx}{dt}\text{ (exists) } \in  {C}_{T }^{\frac{\alpha}{3} }\left( {\mathbb{R},Z}\right) ,}\right. \\
 &\qquad\qquad  \| x{\| }_{Y} =\| x{\| }_{Z,1 + \frac{\alpha}{3} } =  \| x{\| }_{Z,\frac{\alpha}{3} } + \left. {\begin{Vmatrix}\frac{dx}{dt}\end{Vmatrix}}_{Z,\frac{\alpha}{3} }\right\}  ,\\
 &  Y \cap  E = {C}_{T }^{1 + \frac{\alpha}{3} }\left( {\mathbb{R},Z}\right)  \cap {C}_{T }^{\frac{\alpha}{3} }\left( {\mathbb{R},X}\right) \text{ is a Banach space with norm } \| x{\parallel }_{X,\frac{\alpha}{3} } + {\begin{Vmatrix}\frac{dx}{dt}\end{Vmatrix}}_{Z,\frac{\alpha}{3} }.
\end{align*}

Assume that ${\rho_* }$ is the positive $T-$periodic solution given in Theorem \ref{thm:1.2}. Hence \eqref{Q222} has a solution $(\rho_*, \gamma)$ .

Define $G:\tilde{U} \times \tilde{V} \rightarrow W $ by
\begin{equation}\label{GGG}
   G(\tilde{\rho}, \gamma) = \frac{d \rho_*}{d t} +\frac{d\tilde{\rho}}{d t}  -  F(\rho_*+\tilde{\rho},\gamma),
\end{equation}
where
$ \tilde{U} \subset  Y \cap E $ is a sufficiently small neighborhood of $0$ to ensure that \( \rho_*+ \tilde{\rho} >0 \) holds for any $\tilde{\rho} \in \tilde{U} $ and $ \tilde{V} \subset \mathbb{R}$ is an open neighborhood of $\gamma_0$.

\begin{lem}
The mapping $G:\tilde{U} \times \tilde{V} \rightarrow  W $ is well defined and $G\in C^{\infty}(\tilde{U} \times \tilde{V}, W)$.
\end{lem}
\begin{proof}
 The fact that $\rho_* \in C^{1+\frac{\alpha}{3}}{([0,\infty))} $ and $ \Phi\in C^{1+\frac{\alpha}{3}}{([0,\infty))}$ are $T-$periodic, \eqref{f2.6f}, \eqref{Q2.2} and \eqref{rq1} imply the result.
\end{proof}

\vskip 0.3cm

Next, we use the semigroup theory to prove that the operator $ G_{\tilde{\rho}}(0,\gamma)$ is a Fredholm operator with index zero.

\begin{lem}\label{lem2.8}
$ G_{\tilde{\rho}}(0,\gamma)\in  L((Y \cap E),W)$
is a Fredholm operator with index zero.
\end{lem}

\begin{proof}
From \re{GGG}, \re{f2.6f} and \re{Q2.2}, the derivative of $G$ at $0$ is given by
\begin{align}
\nonumber
 D_{\tilde{\rho}} G(0, \gamma)
 = \frac{d }{d t} +\gamma\Psi_1 ({\rho_*(t)})
 +\gamma D_{\tilde{\rho}}\Psi_1 ({\rho_*(t)})[\cdot,{\rho_*(t)}]+ \Phi(t)  D_{\tilde{\rho}}\Psi_2({\rho_*(t)})- D_{\tilde{\rho}}\Psi_3 ({\rho_*(t)}) .
\end{align}
\cite[Page 183]{2009Well} gives
$$ D_{\tilde{\rho}}\Psi_1 ({\rho_*(t)})[\cdot,{\rho_*(t)}],~~  D_{\tilde{\rho}}\Psi_2({\rho_*(t)}) \text{
 and  } D_{\tilde{\rho}}\Psi_3 ({\rho_*(t)}) \in {L}(h^{4+\alpha}(\bar{D}),h^{2+\alpha}(\bar{D}))
 \text{ for any } t\in [0, T].$$
 Then
 \begin{align}\label{compact1}
   \gamma D_{\tilde{\rho}}\Psi_1 ({\rho_*(t)})[\cdot,{\rho_*(t)}]+ \Phi(t)  D_{\tilde{\rho}}\Psi_2({\rho_*(t)}) - D_{\tilde{\rho}}\Psi_3 ({\rho_*(t)})
:  Y \cap E \rightarrow W  \text{ is compact}.
 \end{align}
The term $ \frac{d }{d t} +\gamma\Psi_1 ({\rho_*(t)})  $ is the main component of $ D_{\tilde{\rho}} G(0, \gamma) $. We shall study this operator in detail.

\cite[Proposition I.13.1]{Kielh2012Bifurcation}  showed that under the three assumptions
\begin{align}\label{fffgg1}
\begin{array}{l}
{A}_{0}(t) \text{ generates a holomorphic semigroup for each fixed }
 t \in [0,T],
 \text{ and this family of } \\  \text{semigroups,
  in turn, generates a
``fundamental solution" }
 {U}_{0}( t,\tau)  \in  L( Z,Z) \text{ for } 0 \leq \\ \tau  \leq  t \leq  T  \text{ such that } {U}_{0}( t,t)  = I, \;{U}_{0}( t,\tau ) {U}_{0}( \tau ,s)  = {U}_{0}( t,s), \text{ and any
 solution of }\\ \frac{dx}{dt} = {A}_{0}( t) x  \text{ is given by  } x( t)  = {U}_{0}( t,0) x( 0).
\end{array}
\end{align}
\begin{align}\label{fffgg2}
\begin{array}{l}
{U}_{0}( t,\tau )  \in  L( Z,Z)   \text{ is compact for } 0 \leq  \tau  < t \leq  T.
\end{array}
\end{align}
\begin{align}\label{fffgg3}
\begin{array}{l}
\text{For } f \in C^\beta([0, T], Z) \text{ and } \hat{\varphi} \in Z,  \text{ the solution of } \frac{d x}{d t}=A_0(t) x+f \text{ and } x(0)=\hat{\varphi} \text{ is given } \\\text{by }
x(t)=U_0(t, 0) \hat{\varphi}+\int_0^t U_0(t, s) f(s) d s
 \text{ and } x \in C^\beta([\varepsilon, T], X) \text{ for any } \varepsilon>0,
\end{array}
\end{align}
the operator \( {\widehat{J}}_{0}= \frac{d}{dt} - A_0(t) : Y \cap E \rightarrow W \) is a Fredholm operator of index zero,
where  $W=C_{2 \pi}^\beta(\mathbb{R}, Z), E=C_{2 \pi}^\beta(\mathbb{R}, X)$, and $Y=$ $C_{2 \pi}^{1+\beta}(\mathbb{R}, Z)$.

Take
\begin{align}\nonumber
A_0(t)= -\gamma\Psi_1 ({\rho_*(t)}).
\end{align}

\cite[Theorem 3]{2009Well} or the similar proof of \cite[the proof of Theorem 1]{2001Classical}  implies that $\Psi_1\left(\hat{\rho}\right)
 \in{H}(h^{4+\alpha}_{2\pi}(\bar{D}),h^{1+\alpha}_{2\pi}(\bar{D}))$. Here, ${H}(E_1,E_0)$ is the set of all $A\in{L}(E_1,E_0)$ such that $-A$, considered as a linear operator in $ E_0$ with domain $E_1$, is the infinitesimal generator of a strongly continuous analytic semigroup $\left\{e^{-{tA}}\:;\:t\geq0\:\right\}$ on $E_0$, that is, in ${L}(E_0).$

Hence,  for every
 $t\in [0, T]$, $-\Psi_1\left(\rho_{*}(t)\right)
 $ has a constant domain $h^{4+\alpha}_{2\pi}(\bar{D})$ and $\Psi_1\left(\rho_{*}(t)\right)
 \in {H}(h^{4+\alpha}_{2\pi}(\bar{D}),$ $h^{1+\alpha}_{2\pi}(\bar{D}))$. Together with a  H\"{o}lder continuous of $\Psi_1\left(\rho_{*}(t)\right)$ with the topology of ${L}(h^{4+\alpha}_{2\pi}(\bar{D}),h^{1+\alpha}_{2\pi}(\bar{D}))$,
\eqref{fffgg1} and \eqref{fffgg3} with $\beta=\frac{\alpha}{3}$ are proved in \cite[Chapter II]{Amann1995}. Below is a detailed explanation.

The fact that $\Psi_1\left(\rho_{*}(t)\right)
 \in {H}(h^{4+\alpha}_{2\pi}(\bar{D}),h^{1+\alpha}_{2\pi}(\bar{D}))$ for every
 $t\in [0, T]$, \eqref{rq1} and $\rho_{*}(t)\in C^{\frac{\alpha}{3}}$
  imply $\Psi_1(\rho_*(t))\! \in\! C^{\frac{\alpha}{3}} \!\!\left([0, T], {L}\left(h^{4+\alpha}_{2\pi}\left(\bar{D}\right)\!, h^{1+\alpha}_{2\pi}\left(\bar{D}\right)\right) \right) $,
  i.e., $\Psi_1(\rho_*(t)) \!\in \!  C^{\frac{\alpha}{3}} \left([0, T], H\!\!\left(h^{4+\alpha}_{2\pi}\left(\bar{D}\right)\!, h^{1+\alpha}_{2\pi}\left(\bar{D}\right)\right) \right)$.
\cite[Chapter II Corollary 4.4.2]{Amann1995} implies that \eqref{fffgg1} holds and
\cite[Chapter II Theorem 1.2.1]{Amann1995} shows that \eqref{fffgg3} is true.

The compactness of \eqref{fffgg2} is given
by a compact embedding $ h^{4+\alpha}(\bar{D})\subset h^{1+\alpha}(\bar{D})$.

Hence we get that $\frac{d }{d t} +\gamma\Psi_1(\rho_*(t)) : Y \cap  E  \rightarrow W  $ is a Fredholm operator of index zero.
Together with \re{compact1}, we find that $G_{\tilde{\rho}}(0,\gamma)$ is a Fredholm operator with index zero.
\end{proof}

In order to apply Crandall-Rabinowitz theorem, we need to compute the derivative of $G$ at $0$.
 Since the derivatives of $R$, $ \mathcal{T}$ and $S$ are difficult to calculate, it is  complex to use them to derive the derivative of \( G \).  However, calculating the derivative of $G$ via the free boundary is comparatively simpler.  Through  Lemma \ref{2.3ss} and Lemma \ref{2.4ss}, we derive the derivative expression in terms of the solution of the free boundary problem.

\re{GGG} is equivalent to
\begin{equation}\label{GGG2}
   G(\tilde{\rho}, \gamma) = \frac{d \rho_*}{d t} +\frac{d\tilde{\rho}}{d t}  +{B}(\rho_*(t)+\tilde{\rho}(t)) v,
\end{equation}
where $v$ is defined by $\re{fq2.1}_1$--$\re{fq2.1}_6$.

Now we linearize problem $\re{fq2.1}_1$--$\re{fq2.1}_6$ and compute $  G_{\tilde{\rho}}(0, \gamma)$ by $(u , v )$ through the G$\hat{a}$teaux derivative.  We take $\tilde{\rho}(x_1, x_2, t)=\varepsilon S(x_1, x_2, t)$ and collect the coefficients of $\epsilon$ of $(u , v)$.
Assume that $(u , v)$  has the following expressions:
\begin{align}\nonumber
        &u (x_1, x_2,y', t) = u_*(y',t) + \epsilon u_1(x_1, x_2,y', t) + O(\epsilon^2),\\
        &v (x_1, x_2,y', t) = v_* (y',t)+ \epsilon v_1(x_1, x_2,y', t) + O(\epsilon^2).\label{vp1}
\end{align}
Since the following proof involves two regions $\Omega$ and $\Omega_{\rho(t)}$ and a transformation between them, we denote by $(x_1,x_2,y')$ the point in the first region and by $(x_1,x_2,y)$ the point in the second region.

\begin{lem}\label{2.3ss}
\begin{equation}\label{eq2.24a}
    G_{\tilde{\rho}}(0,\gamma)[S](x_{1}, x_{2},t) = \frac{d S}{d t}(x_{1}, x_{2},t) -\frac{\partial v_*}{\partial y' }(1,t) \frac{S(x_{1}, x_{2},t)}{\rho_*^2(t)} +\frac{\partial v_1}{\partial y'} (x_1, x_2,1, t) \frac{1}{\rho_*(t)} .
\end{equation}
\end{lem}

\begin{proof} $\eqref{GGG2}$ and \re{vp1} imply
\begin{align}\nonumber
   &G(\varepsilon S(x_1, x_2, t), \gamma)
   \\\label{c2.34s}
   &=  \frac{d \rho_*}{d t}(x_{1}, x_{2},t) +  \varepsilon \frac{d S}{d t}(x_{1}, x_{2},t) +  {B}(\rho_*(t)+\varepsilon S( t)) [v_* (y',t)+ \epsilon v_1(x_1, x_2,y', t) ] + O(\epsilon^2) .
\end{align}
We collect the coefficients of $\epsilon$ of $ {B}(\rho_*(t)+\varepsilon S( t)) [v_* (y',t)+ \epsilon v_1(x_1, x_2,y', t) ]$ from ${B}(\rho_*(t)+\varepsilon S( t))$ $ [v_* (y',t)]$ and ${B}(\rho_*(t)) \left[v_* (y',t)+ \epsilon v_1(x_1, x_2,y', t) \right]  $.

Since
\begin{align*}
{B}(\rho(t)) [v_* (y',t)]
&=\left(\operatorname{tr}  {\nabla \left( {{\theta }_{ * }^{\rho(t) }v_{ * }}\right) , \bm{v}_{\rho(t)}}\right) \left( {\theta }_{\rho(t) }\right)    \\
 &= \left(  \operatorname{tr}\nabla\left [v_* \left(\frac{y}{\rho(x_1, x_2,t) },t\right) \right] , (-\rho_{x_1}, -\rho_{x_2}, 1 ) \right)
 \left( {\theta }_{\rho(t) }\right)\\
 &= \left[ \frac{\partial v_*}{\partial y'}\left(\frac{y}{\rho(x_1, x_2,t) },t\right) \frac{1}{\rho(x_1, x_2,t)}
  + \frac{\partial v_*}{\partial y'}\left(\frac{y}{\rho(x_1, x_2,t) },t\right) \frac{y\rho_{x_1}^2}{\rho^2(x_1, x_2,t)} \right. \\
 &\qquad\qquad\left.+\frac{\partial v_*}{\partial y'}\left(\frac{y}{\rho(x_1, x_2,t) },t\right)\frac{y\rho_{x_2}^2}{\rho^2(x_1, x_2,t)} \right]  \bigg |_{y=\rho(x_1, x_2,t)}
 \left( {\theta }_{\rho(t) }\right)\\
 &=  \frac{\partial v_*}{\partial y' }(1,t) \frac{1}{\rho(x_1, x_2,t)}  + \frac{\partial v_*}{\partial x_1}(1,t)   \frac{\rho_{x_1}^2}{\rho(x_1, x_2,t)}
 +\frac{\partial v_*}{\partial x_2}(1,t)  \frac{\rho_{x_2}^2}{\rho(x_1, x_2,t)}.
\end{align*}
Taking ${\rho}( t)=\rho_*(t)+\varepsilon S( t)$ in the above equality, the coefficients of $\epsilon$ of  ${B}(\rho_*(t)+\varepsilon S( t))$ $ [v_* (y',t)]$ is
\begin{align}\label{eqa2}
&-\frac{\partial v_*}{\partial y' }(1,t) \frac{S(x_{1}, x_{2},t)}{\rho_*^2(t)} .
\end{align}

Also,
\begin{align}\nonumber
&{B}(\rho_*(t)) \left[v_* (y',t)+ \epsilon v_1(x_1, x_2,y', t) \right]\\\nonumber
&=\left( \operatorname{tr} {\nabla \left( {{\theta }_{ * }^{\rho_*(t) }v}\right) ,\bm{v}_{\rho_*(t) }}\right)  \left( {\theta }_{\rho_*(t) }\right)    \\ \nonumber
  &=  \left(\operatorname{tr} \nabla  \left[v_*  \left(\frac{y}{\rho_*(t)} ,t\right)+ \epsilon v_1 \left(x_1, x_2, \frac{y}{\rho_*(t)} , t\right) \right]  , (0, 0, 1 )
 \right)
  \left( {\theta }_{\rho_* (t)}\right)\\\nonumber
   &= \left( \frac{\partial v_*}{\partial y'} \left(\frac{y}{\rho_*(t)} ,t\right)  \frac{1}{\rho_*(t)} + \epsilon \frac{\partial v_1}{\partial y'}  \left(x_1, x_2,\frac{y}{\rho_*(t)}, t\right)  \frac{1}{\rho_*(t)} + O\left(\varepsilon^{2} \right)  \right)  \bigg |_{y=\rho_*(t)}
  \left( {\theta }_{\rho_* (t)}\right)\\\label{eqa1}
&=\frac{\partial v_*}{\partial y'}(1 ,t)  \frac{1}{\rho_*(t)} + \epsilon \frac{\partial v_1}{\partial y'} (x_1, x_2,1, t)  \frac{1}{\rho_*(t)} + O\left(\varepsilon^{2} \right).
\end{align}
\re{c2.34s}, \re{eqa2} and \re{eqa1} lead that \re{eq2.24a} holds.
\end{proof}

A third equivalent form of \re{GGG} and \re{GGG2}, derived from the free boundary problem, is given by
\begin{equation}\label{GGG3}
G(\tilde{\rho}, \gamma) = \frac{d \rho_*}{d t} +\frac{d\tilde{\rho}}{d t} +( \nabla p, \bm{v}_{{\rho_*(t)+ \tilde{\rho}(t) }})\bigg |_{\Gamma_{\tilde{\rho}(t)}},
\end{equation}
where $p$ is
the solution of the following system:
\begin{align}\label{8.2}
&- \Delta\sigma + \sigma = 0 \hspace{1em}&& \text{in }{U}_{\tilde{\rho}},\\
\label{8.3}
&\left.\displaystyle\frac{\partial \sigma}{\partial y}\right|_{\, \Gamma_0} =0, \mm \sigma \Big|_{\,\Gamma_{\tilde{\rho} (t)} } =  \Phi(t), && t>0,\\
\label{8.4}
&-\Delta p = \mu(\sigma-\tilde{\sigma})\hspace{1em}\; &&\text{in }{U}_{\tilde{\rho}}, \\
\label{8.5}
&  \left.\displaystyle\frac{\partial p}{\partial y} \right|_{\, \Gamma_0} =0,\mm  p \Big|_{\,\Gamma_{\tilde{\rho}(t)} }= \gamma\kappa , && t>0.
\end{align}
Here ${U}_{\tilde{\rho}}=\{(x_{1}, x_{2},y,t)~|~ (x_{1}, x_{2},y)\in {\Omega}_{\rho_*(t)+\tilde{\rho}(t)},~t\in  [0,T]\}$ and ${\Gamma}_{\tilde{\rho}(t)}=\{(x_{1}, x_{2},y)~|~ y={\rho_*(t)+\tilde{\rho}(x_{1}, x_{2},t)}\}$.

We take $\tilde{\rho}(x_1, x_2, t)=\varepsilon S(x_1, x_2, t)$ and collect the coefficients of $\epsilon$ of $(\sigma , p)$. Assume that $(\sigma , p)$  has the following expressions:
\begin{align}
        &\sigma (x_1, x_2,y, t) = \sigma_*(y,t) + \epsilon w(x_1, x_2,y, t) + O(\epsilon^2),\label{sigma1}\\ \label{p1}
        &p (x_1, x_2,y, t) = p_* (y,t)+ \epsilon q(x_1, x_2,y, t) + O(\epsilon^2).
\end{align}

Substituting \re{sigma1}--\re{p1}
into \eqref{8.2}--\eqref{8.5},
we obtain that $(w, q)$ satisfies
\begin{eqnarray}
&&\label{4.9aa}\left \{
\begin{split}
&\Delta w(x_{1}, x_{2},y,t) = w(x_{1}, x_{2},y,t), \hspace{3em} 0<y< \rho_{\ast}(t),\\
& \dis\frac{\partial w}{\partial y}(x_{1}, x_{2},0,t)=0,\qquad w(x_{1}, x_{2},\rho_*(t),t)=-\dis\frac{\partial \sigma_*}{\partial y}(\rho_*(t) ,t)S(x_{1}, x_{2},t),
\end{split}
\right.\\
&&\label{4.18}\left \{
\begin{split}
&- \Delta q(x_{1}, x_{2},y,t) =\mu w(x_{1}, x_{2},y,t)
, \hspace{3em} 0<y< \rho_{\ast}(t), \\
&\dis\frac{\partial q}{\partial y}(x_{1}, x_{2},0,t)=0,\qquad q(x_{1}, x_{2},\rho_{\ast}(t),t)=-\dis \frac{\gamma}{2}
(S_{x_{1}x_{1}}+S_{x_{2}x_{2}})-\dis\frac{\partial p_*}{\partial y}(\rho_{\ast}(t) ,t)S(x_{1}, x_{2},t).
\end{split}
\right.
\end{eqnarray}

From \eqref{t1}, \re{vp1}
and \eqref{p1}, we get
\begin{align*}
&v_*(y',t) + \epsilon v_1(x_1, x_2,y', t) + O(\epsilon^2)\\
&= p_{*}( y'  [\rho_*(t)+\epsilon  S(x_1, x_2, t)],t) + \varepsilon q\left( x_1, x_2, y'  [\rho_*(t)+\epsilon  S(x_1, x_2, t)] ,t\right) + O(\epsilon^2).
\end{align*}
Taking $\epsilon=0$
, we have
\begin{align*}
v_*(y',t)= p_*( \rho_*(t)y' , t) .
\end{align*}
Hence
\begin{align}\label{e2.35aa}
\frac{\partial v_*}{\partial y' }(1,t) = \frac{\partial p_*}{\partial y}( \rho_*(t) ,t)  {\rho_*(t)}.
\end{align}

Calculating \( u_1 \) and \( v_1 \) from \eqref{fq2.1} is quite complicated. Instead, we relate them to \( w \) and \( q \), which are relatively easier to solve.

Then
\begin{align*}
&v_{1} (x_1, x_2,y',t ) \\
&=  {\left. \frac{d}{d\varepsilon }\right| }_{\varepsilon  = 0}\Big\{ p_{*}( y'   [\rho_*(t)+\epsilon  S(x_1, x_2, t)],t) + \varepsilon q\left( x_1, x_2, y'   [\rho_*(t)+\epsilon  S(x_1, x_2, t)],t \right)\Big\} \\
& = \frac{\partial p_*}{\partial y}( y'  \rho_*(t) ,t) y' S(x_{1}, x_{2},t)
+q(x_1,x_2,  y'  \rho_*(t),t).
\end{align*}
Hence
\begin{align*}
 & \frac{\partial v_1}{\partial y'} (x_1, x_2,y', t)  =   \frac{\partial^2 p_*}{\partial y^2} (y'  \rho_*(t) ,t) y' S(x_{1}, x_{2},t)\rho_*(t) \\
&\qquad \qquad+ \frac{\partial p_*}{\partial y}( y'  \rho_*(t) ,t)  S(x_{1}, x_{2},t) +\frac{\partial q}{\partial y}(x_1,x_2,  y'  \rho_*(t),t) \rho_*(t) ,
\end{align*}
and
\begin{align*}
 & \frac{\partial v_1}{\partial y'} (x_1, x_2,1, t)   =   \frac{\partial^2 p_*}{\partial y^2}(\rho_*(t) ,t)  S(x_{1}, x_{2},t) \rho_*(t)
+ \frac{\partial p_*}{\partial y}( \rho_*(t) ,t)  S(x_{1}, x_{2},t) \\
&\qquad+\frac{\partial q}{\partial y}(x_1,x_2,  \rho_*(t),t)  \rho_*(t).
\end{align*}

Together with \eqref{eq2.24a} and \re{e2.35aa}, we have the following lemma.
\begin{lem}\label{2.4ss}
\begin{equation}\label{2.25a}
    G_{\tilde{\rho}}(0,\gamma)[S](x_{1}, x_{2},t) = \frac{d S}{d t}(x_{1}, x_{2},t)+ \frac{\partial ^2 p_*}{\partial y^2} (\rho_*(t),t) S(x_{1}, x_{2},t)+ \frac{\partial q}{\partial y}(x_{1}, x_{2},\rho_*(t),t) .
\end{equation}
\end{lem}

\subsection{Symmetric properties}
{In this subsection, we shall show that $G(\cdot,\gamma)$ maintains periodicity and symmetry through
\re{GGG3}.}
For integers $n_1, m_1\geq 1$, we define the following periodic and symmetric conditions:
\begin{align}\label{5.5a}
 u\left( x_1,x_2,t+T\right)&=u(x_1,x_2,t),\\\label{5.1}
u\left(\frac{2 \pi}{n_1}+ x_1 ,x_2,t\right)&=u(x_1,x_2,t),\\\label{5.2}
  u\left( x_1,\frac{2 \pi}{m_1}+x_2,t\right)&=u(x_1,x_2,t),\\\label{1even}
u\left(- x_1,x_2,t\right)&=u(x_1,x_2,t),\\\label{2even}
u\left( x_1,-x_2,t\right)&=u(x_1,x_2,t),\\
  \label{5.3a}
u\left( \frac{ \pi}{n_1}-x_1,\frac{ \pi}{m_1}-x_2,t\right)&=u(x_1,x_2,t), \\
\label{5.5}
 u\left( x_1,x_2,t\right)&=u(x_2,x_1,t).
\end{align}

Set
\begin{align*}
& (Y \cap E)_{n_1,m_1}=\{ u\in Y \cap E\,|\,   u \text{ satisfies } \re{5.5a}, \eqref{5.1}, \eqref{5.2}, \eqref{1even}, \eqref{2even} \text{ and } \eqref{5.3a} \},\\
& (Y \cap E)_{n_1,0}=\{ u\in Y \cap E\,|\,  u \text{ satisfies } \re{5.5a}, \eqref{5.1}, \eqref{1even} \text{ and }  u \text{ is independent of }\\
&\mm\mm\mm x_2 \},\\
& (Y \cap E)_{n_1, +}=\{ u \in Y \cap E\,|\,   u  \text{ satisfies } \re{5.5a}, \eqref{5.1}, \eqref{5.2} \text{ with } m_1=n_1, \eqref{1even}, \eqref{2even} \\
&\mm\mm\mm\text{ and } \eqref{5.5} \},\\
& W_{n_1,m_1},  W_{n_1,0} \text{ and } W_{n_1, +}  \text{ analogously}.
\end{align*}
We shall show that $G(\cdot,\gamma)$ preserves periodicity and symmetry in the following lemmas.

\begin{lem}\label{prop} If $ \Phi\in C^{1+\frac{\alpha}{3}}{([0,\infty))}$ is T periodic and $\tilde{\rho}\in Y \cap E$ satisfies \re{5.5a},  then $G(\cdot,\gamma)$
also satisfies \re{5.5a}. Precisely, the solution $(\sigma, p)$ of system \re{8.2}--\re{8.5}
satisfies
\begin{align}\label{eqnprop}
    \sigma(x_1,x_2,y, t+T)=\sigma(x_1,x_2,y,t),\mm  p(x_1,x_2,y,t+T)=p(x_1,x_2,y,t).
\end{align}
\end{lem}
\begin{proof}
Replacing $t$ with $t+T$  in  system \eqref{8.2}--\eqref{8.5}, we obtain
\begin{eqnarray}\nonumber
&&\label{bar1}
\left \{
\begin{split}
&-\Delta \sigma(x_1,x_2,y, t+T) +\sigma(x_1,x_2,y, t+T) = 0  \mm\;\;\,
\,&\text{in }{\Theta}_{\tilde{\rho}},\\
&\displaystyle\frac{\partial \sigma}{\partial y}(x_1,x_2,0,t+T) =0,\mm\sigma(x_1,x_2, \rho_*(t+T)+\tilde{\rho}(x_1,x_2,t+T), t+T) = \Phi(t+T),\\
&-\Delta p(x_1,x_2,y, t+T) = \mu(\sigma(x_1,x_2,y, t+T)-\tilde{\sigma})\mm\; &\text{in }\Theta_{\tilde{\rho}},\\
&\displaystyle\frac{\partial p}{\partial y}(x_1,x_2,0,t+T) =0,\mm p(x_1,x_2, \rho_*(t+T)+\tilde{\rho}(x_1,x_2,t+T), t+T) = \gamma\kappa
,
\end{split}
\right.
\end{eqnarray}
where $\Theta_{\tilde{\rho}}$ is a domain with boundary $\Gamma_0\cup \tilde{\Gamma}_{\tilde{\rho}(t)}$ and $ \tilde{\Gamma}_{\tilde{\rho}(t)}=\{y\,|\, \, y=\rho_*(t+T)+\tilde{\rho}(x_1, x_2, t+T)
\} $.
The fact that $\tilde{\rho}$ satisfied \re{5.5a} and $ \rho_*$ is T periodic  imply $\Theta_{\tilde{\rho}} \equiv \Omega_{\tilde{\rho}}$.
Define $\check{\sigma}(x_1, x_2,y,t) = \sigma(x_1,x_2,y, t+T)$ and $\check{p}(x_1, x_2,y,t) = p(x_1,x_2,y, t+T)$.
Applying \re{5.5a} and that $ \Phi$ is T periodic, we obtain that $\check{\sigma}$ satisfies \eqref{8.2}--\eqref{8.3} and $\check{p}$ satisfies \eqref{8.4}--\eqref{8.5}. The uniqueness of the solution for \eqref{8.2}--\eqref{8.5} implies
$\check{\sigma}(x_1, x_2,y,t)= \sigma(x_1,x_2,y,t)$ and $\check{p}(x_1, x_2,y,t)= p(x_1,x_2,y,t)$, i.e., $\eqref{eqnprop}$ holds.
We complete the proof.
\end{proof}
Similarly, we obtain the following lemmas.

\begin{lem}\label{prop1}
If $ \Phi\in C^{1+\frac{\alpha}{3}}{([0,\infty))}$ is T periodic and $\tilde{\rho}\in Y \cap E$ satisfies \re{5.1},
then $G(\cdot,\gamma)$ also satisfies \re{5.1}.  Precisely,  the solution $(\sigma, p)$ of system \re{8.2}--\re{8.5} satisfies
\begin{align}\nonumber
    \sigma(\frac{2 \pi}{n_1}+x_1,x_2,y,t )=\sigma(x_1,x_2,y,t),\mm  p(\frac{ 2\pi}{n_1}+x_1,x_2,y,t )=p(x_1,x_2,y,t).
\end{align}
\end{lem}
\begin{lem}\label{props2}
If $ \Phi\in C^{1+\frac{\alpha}{3}}{([0,\infty))}$ is T periodic and $\tilde{\rho}\in Y \cap E$ satisfies \re{5.2},
then $G(\cdot,\gamma)$
also satisfies \re{5.2}. Precisely,  the solution $(\sigma, p)$ of   system \re{8.2}--\re{8.5} satisfies
\begin{align*}
  \sigma(x_1, \frac{2\pi}{m_1}{ +}x_2,y,t)=\sigma(x_1,x_2,y,t),\mm p(x_1, \frac{2\pi}{m_1}{ +}x_2,y,t)=p(x_1,x_2,y,t).
\end{align*}

\end{lem}

\begin{lem}\label{props3}
If $ \Phi\in C^{1+\frac{\alpha}{3}}{([0,\infty))}$ is T periodic and $\tilde{\rho}\in Y \cap E$ satisfies \re{1even},
then $G(\cdot,\gamma)$
also satisfies \re{1even}. Precisely,  the solution $(\sigma, p)$ of   system \re{8.2}--\re{8.5} satisfies
\begin{align*}
    \sigma(-x_1,x_2,y,t )=\sigma(x_1,x_2,y,t),\mm  p(-x_1,x_2,y,t )=p(x_1,x_2,y,t).
\end{align*}
\end{lem}

\begin{lem}\label{props4}
If $ \Phi\in C^{1+\frac{\alpha}{3}}{([0,\infty))}$ is T periodic and $\tilde{\rho}\in Y \cap E$ satisfies \re{2even},
then $G(\cdot,\gamma)$
also satisfies \re{2even}. Precisely,   the solution $(\sigma, p)$ of   system \re{8.2}--\re{8.5} satisfies
\begin{align*}
  \sigma(x_1, -x_2,y,t)=\sigma(x_1,x_2,y,t),\mm p(x_1, -x_2,y,t)=p(x_1,x_2,y,t).
\end{align*}
\end{lem}

\begin{lem}\label{props4a}
If $ \Phi\in C^{1+\frac{\alpha}{3}}{([0,\infty))}$ is T periodic and $\tilde{\rho}\in Y \cap E$ satisfies \re{5.3a},
then $G(\cdot,\gamma)$
also satisfies \re{5.3a}. Precisely,   the solution $(\sigma, p)$ of   system \re{8.2}--\re{8.5} satisfies
\begin{align*}
    \sigma(\frac{ \pi}{n_1}-x_1,\frac{ \pi}{m_1}-x_2,y,t )=\sigma(x_1,x_2,y,t),\mm  p(\frac{ \pi}{n_1}-x_1,\frac{ \pi}{m_1}-x_2,y,t )=p(x_1,x_2,y,t).
\end{align*}
\end{lem}

\begin{lem}\label{prop5}
If $ \Phi\in C^{1+\frac{\alpha}{3}}{([0,\infty))}$ is T periodic and $\tilde{\rho}\in Y \cap E$ satisfies \re{5.5},
then $G(\cdot,\gamma)$
also satisfies \re{5.5}. Precisely,   the solution $(\sigma, p)$ of  system \re{8.2}--\re{8.5} satisfies
\begin{align*}
    \sigma(x_1,x_2,y,t )=\sigma(x_2,x_1,y,t),\mm  p(x_1,x_2,y,t )=p(x_2,x_1,y,t).
\end{align*}
\end{lem}

Lemmas \ref{prop}--\ref{prop5} imply the following lemma.
\begin{lem}\label{5.7}
$G(\cdot,\gamma)$  maps $ (Y \cap E)_{n_1,m_1}$ into $W_{n_1,m_1}$, or $(Y \cap E)_{n_1,0}$ into $W_{n_1,0}$, or $(Y \cap E)_{n_1,+}$ into $W_{n_1,+}$.
\end{lem}

\subsection{Bifurcations - Proof of Theorem \ref{result1}}
At first, we describe $\mathrm{Ker}[G_{\tilde{\rho}}(0,\gamma)]$.
We find the solution in the following form
\begin{eqnarray*}
&& w(x_{1}, x_{2},y,t) =
w_{n,m}(y,t)\cos(nx_{1})\cos(mx_{2}),\\
&&  q(x_{1}, x_{2},y,t) =
  q_{n,m}(y,t)\cos(nx_{1})\cos( mx_{2}),\\
&& S(x_{1}, x_{2},t) =
  S_{n,m}(t)\cos(nx_{1})\cos(  mx_{2}).
\end{eqnarray*}
  We have only performed the basis expansion in space, not in time. In terms of time, we aim to find a function $S_{n,m}(t)$ and a constant $\gamma_{j}$ such that $S_{n,m}(t)\cos(nx_{1})\cos( mx_{2}) \in \mathrm{Ker}[G_{\tilde{\rho}}(0,\gamma_{j})]$.

By \re{4.9aa} and  \re{4.18}, we obtain that $w_{n,m}$, and $q_{n, m}$ satisfy
\begin{eqnarray}
&&\label{4.20}\left \{
\begin{array}{lr}
- \dis\frac{\partial^2w_{n,m}}{\partial y^2}(y,t) +(n^{2}+m^{2}+1)w_{n,m}(y,t) = 0,\hspace{2em}0<y<\rho_{\ast}(t),\\
\dis\frac{\partial w_{n,m}}{\partial y}(0,t) =0,\m w_{n,m}(\rho_*(t),t) = - \dis\frac{\partial \sigma_*}{\partial y}(\rho_*(t),t)S_{n,m}(t)
,
\end{array}
\right.\\
&&\label{4.21a}\left \{
\begin{array}{lr}
-\dis\frac{\partial^2 q_{n,m}}{\partial y^2}(y,t)+ (n^{2}+m^{2})q_{n,m}(y,t) = \mu w_{n,m}    (y,t),  \hspace{2em}0<y<\rho_{\ast}(t),\\
\dis\frac{\partial q_{n,m}}{\partial y}(0,t) =0,\m  q_{n,m}(\rho_*(t),t) =\dis \frac\gamma 2(n^{2}+m^{2})S_{n,m}(t)-\dis\frac{\partial p_*}{\partial y}(\rho_*(t),t)S_{n,m}(t)
.
\end{array}
\right.
\end{eqnarray}
Solving \re{4.20} and \re{4.21a}, we get
\begin{equation*}
\begin{aligned}
w_{n,m}(y,t) &= -\frac{\partial \sigma_*}{\partial y}(\rho_*(t),t) S_{n,m}(t)\frac{\cosh(\sqrt{1+n^{2}+m^{2}}y)}{\cosh(\sqrt{1+n^{2}+m^{2}}\rho_*(t))}, \\
q_{n,m}(y,t)
&=\mu\displaystyle\frac{\partial  \sigma_*}{\partial y}(\rho_*(t),t)S_{n,m}(t) \frac{\cosh (\sqrt{1+n^{2}+m^{2}}y) }{\cosh (\sqrt{1+n^{2}+m^{2}} \rho_*(t))}\\
&\mm+ \displaystyle \Big[\frac{\gamma}{2}(n^{2}+m^{2})-\mu\displaystyle\frac{\partial  \sigma_*}{\partial y}(\rho_*(t),t)-\displaystyle\frac{\partial  p_*}{\partial y}(\rho_*(t),t)\Big] S_{n,m}(t) \frac{\cosh (\sqrt{n^{2}+m^{2}}y)}{\cosh (\sqrt{n^{2}+m^{2}}\rho_*(t))} .
\end{aligned}
\end{equation*}
Then
\begin{equation*}
\begin{aligned}
\dis\frac{\partial q_{n,m}}{\partial y}(\rho_*(t),t)
&=\mu\displaystyle\frac{\partial  \sigma_*}{\partial y}(\rho_*(t),t)S_{n,m}(t) \sqrt{1+n^{2}+m^{2}}\tanh (\sqrt{1+n^{2}+m^{2}} \rho_*(t))\\
&\mm+ \displaystyle \Big[\frac{\gamma}{2}(n^{2}+m^{2})-\mu\displaystyle\frac{\partial  \sigma_*}{\partial y}(\rho_*(t),t)-\displaystyle\frac{\partial  p_*}{\partial y}(\rho_*(t),t)\Big] S_{n,m}(t)\\
 &\mm\mm\cdot \sqrt{n^{2}+m^{2}} \tanh (\sqrt{n^{2}+m^{2}}\rho_*(t)).
\end{aligned}
\end{equation*}

Combined with \re{2.25a},
$\frac{\partial  \sigma_*}{\partial y}(\rho_*(t),t)=\Phi(t)\tanh(\rho_*(t))$, $\frac{\partial p_*}{\partial y}(\rho_*(t), t)=\mu \tilde{\sigma} \rho_*(t)-\mu \Phi(t) \tanh(\rho_*(t))$ and $\frac{\partial^{2 }p_*}{\partial y^{2}}(\rho_*(t), t)=\mu \tilde{\sigma}-\mu \Phi(t)$, we obtain
\begin{equation}\label{fff}
    \begin{split}
       &G_{\tilde{\rho}}(0,\gamma) [S_{n,m}(t)\cos(nx_{1})\cos(  mx_{2})]\\
       &=
         \frac{d S_{n,m}}{d t}\cos(nx_{1})\cos(m x_{2})\\
     &\mm+S_{n,m}(t)\Big\{  \mu \tilde{\sigma}-\mu \Phi(t)  +\mu\Phi(t)\tanh(\rho_*(t))\sqrt{1+n^{2}+m^{2}}\tanh(\sqrt{1+n^{2}+m^{2}}\rho_*(t))\\
       &\mm-\mu\tilde{\sigma} \rho_*(t)\sqrt{n^{2}+m^{2}}\tanh (\sqrt{n^{2}+m^{2}}\rho_*(t))\\
&\mm+\frac{\gamma(n^{2}+m^{2})^{\frac{3}{2}}}{2}\tanh (\sqrt{n^{2}+m^{2}}\rho_*(t))\Big \} \cos(nx_{1})\cos(m x_{2}).
    \end{split}
\end{equation}

Hence $ G_{\tilde{\rho}}(0,\gamma)[S_{n,m}(t)\cos(nx_{1})\cos(  mx_{2})]=0$ if and only if
\begin{equation}\nonumber
    \begin{split}
       &\frac{d S_{n,m}}{d t}+S_{n,m}(t)\Big\{  \mu\tilde{\sigma}-\mu \Phi(t)  +\mu\Phi(t)\tanh(\rho_*(t))\sqrt{1+n^{2}+m^{2}}\tanh(\sqrt{1+n^{2}+m^{2}}\rho_*(t))\\
       &-\mu\tilde{\sigma} \rho_*(t)\sqrt{n^{2}+m^{2}}\tanh (\sqrt{n^{2}+m^{2}}\rho_*(t))+\frac{\gamma(n^{2}+m^{2})^{\frac{3}{2}}}{2}\tanh (\sqrt{n^{2}+m^{2}}\rho_*(t))\Big \} =0,
    \end{split}
\end{equation}
i.e., $S_{n,m}(t)=e^{-\int_{0}^{t} A_{n^2+m^2,\gamma}(s)ds}$, where $A_{j,\gamma}(t)$ is defined in \eqref{Aj}.
Since $S_{n,m}(t)$ depends only on $n^2+m^2$ and $\gamma$, we denote $S_{n,m}(t)$ by $S_{n^2+m^2, \gamma }(t)$.
By \eqref{1.15}, we have
 \begin{align}\nonumber
&A_{j, \gamma}(t)\\\nonumber
&=\mu\Phi(t) \frac{\tanh(\rho_*(t))}{\rho_*(t)}-\frac{\rho_*'(t)}{\rho_*(t)}
-\mu \Phi(t)  +\mu\Phi(t)\tanh(\rho_*(t)) \sqrt{1+j}\tanh(\sqrt{1+j}\rho_*(t))\\\nonumber
&\mm-\mu \Phi(t)\tanh (\rho_*(t))\sqrt{j}\tanh (\sqrt{j}\rho_*(t))+\rho_*'(t)\sqrt{j}\tanh (\sqrt{j}\rho_*(t))+\frac{\gamma j^{\frac{3}{2}}}{2}\tanh (\sqrt{j}\rho_*(t)).
\end{align}

Since $S_{j, \gamma}(t)$ is $T-$periodic, it follows
\begin{align}\nonumber
 S_{j, \gamma}(t+T)-S_{j, \gamma}(t)
=e^{-\int_{0}^{t} A_{j, \gamma}(s) ds}\big[ e^{-\int_{0}^{T} A_{j, \gamma}(s) ds} -1 \big]=0.
\end{align}
Then $\int_{0}^{T}A_{j, \gamma}(s)ds=0$, which implies
 $\gamma=\gamma_{j}(\rho_*)$ (given by \re{1}).

Next, we discuss the monotonicity of ${ \gamma_{j}}(\rho_*)$.
\begin{lem}\label{lem6.2}
There exists $N_0>0$ such that $ \gamma_{j}(\rho_*)$ is strictly decreasing in $j$ for $j>N_0$.
Also,
\begin{align} \label{2.49}
\lim_{j\rightarrow+\infty}  \gamma_{j}(\rho_*)=0.
\end{align}
\end{lem}
\begin{proof}
We extend the domain of  $j$ in $ \gamma_j$ to be nonnegative real number. Then
\begin{equation}\nonumber
    \frac{1}{\mu}\frac{\partial \gamma_j } {\partial j}=\frac{\displaystyle\frac{\partial  k_1}{\partial j} \; k_2- k_1 \displaystyle\frac{\partial  k_2}{\partial j}   }{ \; k_2^{2}}=:\frac{-\tilde{A}_{j}(\rho_*)}{ \; k_2^{2}}.
\end{equation}
We recall a formula in \cite[(2.3)]{hxh2}
\begin{eqnarray}\label{2.3b}
\tanh z=2z\sum_{\substack{k=1\\  k \ odd}}^\infty\frac{1}{\left(\displaystyle\frac{k\pi}{2}\right)^{2}+z^{2}}.
\end{eqnarray}
Together with  \re{2}, we get
\begin{align}\label{5.25}
&k_1(j)=\int_{0}^{T}\Phi(t)\bigg\{1-\frac{\tanh (\rho_*(t))}{\rho_*(t)}-2 \dis\frac{\tanh (\rho_*(t))}{\rho_*(t) }\sum_{\substack{k=1\\  k \ odd}}^\infty \frac{\Big(\dis\frac{k \pi}{2 \rho_*(t)}\Big)^{2}}{\Big[\Big(\dis\frac{k \pi}{2 \rho_*(t)}\Big)^{2}+1+j\Big]\Big[\Big(\dis\frac{k \pi}{2 \rho_*(t)}\Big)^{2}+j\Big]}\bigg\}dt,\\ \label{5.12a}
 & \frac{\partial  k_1}{\partial j}=\int_{0}^{T}2\Phi(t)\frac{\tanh(\rho_*(t))}{\rho_*(t)}\sum_{\substack{k=1\\  k \ odd}}^\infty\frac{\Big(\displaystyle\frac{k\pi}{2\rho_*(t)}\Big)^{2}\Big[2\Big(\displaystyle\frac{k\pi}{2\rho_*(t)}\Big)^{2}+2j+1\Big]}
   {\Big[\Big(\displaystyle\frac{k\pi}{2\rho_*(t)}\Big)^{2}+1+j\Big]^{2}\Big[\Big(\displaystyle\frac{k\pi}{2\rho_*(t)}\Big)^{2}+j\Big]^{2}}dt.
\end{align}
By \re{3} and \re{2.3b}, we have
\begin{align}\label{k2d}
    \frac{\partial  k_2}{\partial j} & =\int_{0}^{T} \frac{1}{\rho_*(t)}\bigg\{\sum_{\substack{k=1\\  k \ odd}}^\infty\frac{2j}{\Big(\displaystyle\frac{k\pi}{2\rho_*(t)}\Big)^{2}+j}
    -\sum_{\substack{k=1\\  k \ odd}}^\infty\frac{j^2}{\Big[\Big(\displaystyle\frac{k\pi}{2\rho_*(t)}\Big)^{2}+j\Big]^{2}}\bigg\}dt\\
    &\geq\int_{0}^{T} \frac{1}{\rho_*(t)}\sum_{\substack{k=1\\\nonumber  k \ odd}}^\infty\frac{j}{\Big(\displaystyle\frac{k\pi}{2\rho_*(t)}\Big)^{2}+j}dt=\int_{0}^{T}\frac{1}{2} \sqrt{j}\tanh (\sqrt{j}\rho_*(t))dt.
\end{align}

From \re{5.25}, \re{k2d},  \re{5.12a} and the fact that $k_1(j){>}0$ for $j > j_{0}$,
it follows
\begin{align}\label{5.13a}
  \tilde{A}_{j}(\rho_*)&=k_1\frac{\partial  k_2}{\partial j}-   \frac{\partial  k_1}{\partial j} \; k_2\\\nonumber
   &\geq\int_{0}^{T}\Phi(t)\bigg\{1-\frac{\tanh (\rho_*(t))}{\rho_*(t)}-2 \dis\frac{\tanh (\rho_*(t))}{\rho_*(t) }\sum_{\substack{k=1\\  k \ odd}}^\infty \frac{\Big(\dis\frac{k \pi}{2 \rho_*(t)}\Big)^{2}}{\Big[\Big(\dis\frac{k \pi}{2 \rho_*(t)}\Big)^{2}+1+j\Big]\Big[\Big(\dis\frac{k \pi}{2 \rho_*(t)}\Big)^{2}+j\Big]}\bigg\}dt\\\nonumber
&\mm\mm\cdot\int_{0}^{T}\frac{1}{2} \sqrt{j}\tanh (\sqrt{j}\rho_*(t))dt-\int_{0}^{T}\frac12  j^{3/2} \tanh(\sqrt{j}\rho_*(t))dt\\\nonumber
 &\mm\mm\cdot \int_{0}^{T}2\Phi(t)\frac{\tanh (\rho_*(t))}{\rho_*(t)}\sum_{\substack{k=1\\  k \ odd}}^\infty\frac{\Big(\displaystyle\frac{k\pi}{2\rho_*(t)}\Big)^{2}\Big[2\Big(\displaystyle\frac{k\pi}{2\rho_*(t)}\Big)^{2}+2j+1\Big]}
   {\Big[\Big(\displaystyle\frac{k\pi}{2\rho_*(t)}\Big)^{2}+1+j\Big]^{2}\Big[\Big(\displaystyle\frac{k\pi}{2\rho_*(t)}\Big)^{2}+j\Big]^{2}}dt\\\nonumber
   &=\int_{0}^{T}\frac{1}{2}  \sqrt{j}\tanh (\sqrt{j}\rho_*(t))dt\cdot\int_{0}^{T}\Phi(t)\bigg\{1-\frac{\tanh( \rho_*(t))}{\rho_*(t)}\\\nonumber
   &\mm-2 \dis\frac{\tanh (\rho_*(t))}{\rho_*(t) }\sum_{\substack{k=1\\  k \ odd}}^\infty \frac{\Big(\dis\frac{k \pi}{2 \rho_*(t)}\Big)^{2}}{\Big[\Big(\dis\frac{k \pi}{2 \rho_*(t)}\Big)^{2}+1+j\Big]\Big[\Big(\dis\frac{k \pi}{2 \rho_*(t)}\Big)^{2}+j\Big]}\\\nonumber
   &\mm- 2  \frac{\tanh (\rho_*(t))}{\rho_*(t)} \sum_{\substack{k=1\\  k \ odd}}^\infty\frac{\Big(\displaystyle\frac{k\pi}{2\rho_*(t)}\Big)^{2}j\Big[2\Big(\displaystyle\frac{k\pi}{2\rho_*(t)}\Big)^{2}+2j+1\Big]}
   {\Big[\Big(\displaystyle\frac{k\pi}{2\rho_*(t)}\Big)^{2}+1+j\Big]^{2}\Big[\Big(\displaystyle\frac{k\pi}{2\rho_*(t)}\Big)^{2}+j\Big]^{2}} \bigg\}dt\\\nonumber
   &=: \int_{0}^{T}\frac{1}{2}  \sqrt{j}\tanh (\sqrt{j}\rho_*(t))dt\cdot B_{j}(\rho_*).
\end{align}
From  \re{5.13a}, we have
\begin{align}\nonumber
   B_{j}( \rho_*)
   &\geq \int_{0}^{T}\Phi(t)\bigg\{1-\frac{\tanh( \rho_*(t))}{\rho_*(t)}-2 \frac{\tanh (\rho_*(t))}{\rho_*(t)} \sum_{\substack{k=1\\  k \ odd}}^\infty\frac{\Big(\displaystyle\frac{k\pi}{2\rho_*(t)}\Big)^{2}}
   {\Big[\Big(\displaystyle\frac{k\pi}{2\rho_*(t)}\Big)^{2}+1+j\Big]\Big[\Big(\displaystyle\frac{k\pi}{2\rho_*(t)}\Big)^{2}+j\Big]} \\\nonumber
   &\mm-2 \frac{\tanh (\rho_*(t))}{\rho_*(t)} \sum_{\substack{k=1\\  k \ odd}}^\infty\frac{2\Big(\displaystyle\frac{k\pi}{2\rho_*(t)}\Big)^{2}}
   {\Big[\Big(\displaystyle\frac{k\pi}{2\rho_*(t)}\Big)^{2}+1+j\Big]\Big[\Big(\displaystyle\frac{k\pi}{2\rho_*(t)}\Big)^{2}+j\Big]}\bigg\}dt\\\label{5.14a}
    &{ =}\int_{0}^{T}\Phi(t)\bigg\{1-\frac{\tanh( \rho_*(t))}{\rho_*(t)}-3 \frac{\tanh (\rho_*(t))}{\rho_*(t)}2 \sum_{\substack{k=1\\  k \ odd}}^\infty\frac{\Big(\displaystyle\frac{k\pi}{2\rho_*(t)}\Big)^{2}}
   {\Big[\Big(\displaystyle\frac{k\pi}{2\rho_*(t)}\Big)^{2}+1+j\Big]\Big[\Big(\displaystyle\frac{k\pi}{2\rho_*(t)}\Big)^{2}+j\Big]}\bigg\}dt\\\nonumber
   &=\int_{0}^{T}\Phi(t)\bigg\{1-\frac{\tanh(\rho_*(t))}{\rho_*(t)}-3 \tanh( \rho_*(t))\Big[\sqrt{1+j}\tanh(\sqrt{1+j}\rho_*(t) )\\\nonumber
   &\mm-
 \sqrt{j}\tanh(\sqrt{j}\rho_*(t))\Big]\bigg\}dt\\\nonumber
   &=: L_{j} ( \rho_*).
\end{align}
Obviously, $L_{j} (\rho_*)$  is strictly increasing in $j$.
Combining  \re{5.13a} and \re{5.14a}, we get
\begin{equation}\label{5.31aaa}
     \frac{1}{\mu}\frac{\partial \gamma_j}  {\partial j}=\frac{-\tilde{A}_{j}(\rho_*)}{ \; k_2^{2}}\leq\frac{-\int_{0}^{T}\frac{1}{2}  \sqrt{j}\tanh (\sqrt{j}\rho_*(t))dt\cdot B_{j}( \rho_*)}{ \; k_2^{2}}\leq \frac{ -\int_{0}^{T}\frac{1}{2}  \sqrt{j}\tanh (\sqrt{j}\rho_*(t))dt\cdot L_{j} ( \rho_*)}{ \; k_2^{2}}.
\end{equation}

Since
$$
\lim_{j\rightarrow+\infty}L_{j} ( \rho_*)=\int_{0}^{T}\Phi(t)\bigg(1-\frac{\tanh(\rho_*(t))}{\rho_*(t)}\bigg)dt>0,
$$
it follows that there exists $N_0>j_0$ such that $L_{j} ( \rho_*)>0$, when $j>N_0$. Also, applying \re{5.31aaa}, we have  $ \frac{\partial  \gamma_j}{\partial j}<0$ for $j>N_0$, which implies that $ \gamma_{j}(\rho_*)$  is decreasing in $j$ for $j>N_0$. Also, \eqref{2.49} holds.
\end{proof}
\begin{lem}\label{lem5.3}
The monotonicity of $ \gamma_{j}(\rho_*)$ in $j$ changes at most once.
\end{lem}

\begin{proof}
It is sufficient to show that if $\frac{\partial \gamma_j}{\partial j}(\rho_*)=0$, then $ \frac{\partial^{2} \gamma_j}{ \partial j^{2}}(\rho_*)<0$.

From \re{1} and $\frac{\partial \gamma_j}{\partial j}(\rho_*)=0$, we get
\begin{align}\label{eded}
 \frac{1}{\mu}\frac{\partial^{2}  \gamma_j }{\partial j^{2}}(\rho_*)&=\frac{\Big(\displaystyle\frac{\partial^{2} k_1}{\partial j^{2}} k_2+\displaystyle\frac{\partial k_1}{\partial j} \displaystyle\frac{\partial k_2}{\partial j}-\displaystyle\frac{\partial k_1}{\partial j} \displaystyle\frac{\partial k_2}{\partial j}-k_1 \displaystyle\frac{\partial^{2} k_2}{\partial j^{2}}\Big){k^{2}_2}}{{k^{4}_2}}-\frac{2\Big(\displaystyle\frac{\partial k_1}{\partial j}k_2-  k_1 \displaystyle\frac{\partial k_2}{\partial j}    \Big)k_2 \displaystyle\frac{\partial k_2}{\partial j}}{{k^{4}_2}}\\\nonumber
&=\frac{\Big(\displaystyle\frac{\partial^{2} k_1}{\partial j^{2}} k_2-k_1 \displaystyle\frac{\partial^{2} k_2}{\partial j^{2}}\Big)}{{k^{2}_2}}=:\frac{P_{j}(\rho_*)}{{k^{2}_2}}.
\end{align}

Together with \re{5.12a} and \re{k2d},
we obtain
    \begin{align}\nonumber
   \frac{\partial^{2} k_1}{\partial j^{2}}(\rho_*)
&=\int_{0}^{T}2\Phi(t)\frac{\tanh(\rho_*(t))}{\rho_*(t)} \sum_{\substack{k=1\\  k \ odd}}^\infty\bigg\{\frac{2\Big(\frac{k \pi}{2 \rho_*(t)}\Big)^{2}}{\Big[\Big(\frac{k \pi}{2 \rho_*(t)}\Big)^{2}+j+1\Big]^{2}\Big[\Big(\frac{k \pi}{2 \rho_*(t)}\Big)^{2}+j\Big]^{2}}\\\nonumber
   &\mm\mm-\frac{2\Big(\frac{k \pi}{2 \rho_*(t)}\Big)^{2}\Big[2\Big(\frac{k \pi}{2 \rho_*(t)}\Big)^{2}+2 j+1\Big]}{\Big[\Big(\frac{k \pi}{2 \rho_*(t)}\Big)^{2}+j+1\Big]^{3}\Big[\Big(\frac{k \pi}{2 \rho_*(t)}\Big)^{2}+j\Big]^{2}}\\\nonumber
   &\mm\mm-\frac{2\Big(\frac{k \pi}{2 \rho_*(t)}\Big)^{2}\Big[2\Big(\frac{k \pi}{2 \rho_*(t)}\Big)^{2}+2 j+1\Big]}{\Big[\Big(\frac{k \pi}{2 \rho_*(t)}\Big)^{2}+j+1\Big]^{2}\Big[\Big(\frac{k \pi}{2 \rho_*(t)}\Big)^{2}+j\Big]^{3}}\bigg\}d t\\\nonumber
   & =\int_{0}^{T}{ 2}\Phi(t)\frac{\tanh (\rho_*(t))}{\rho_*(t)} \sum_{\substack{k=1\\  k \ odd}}^\infty\frac{\Big(\frac{k \pi}{2 \rho_*(t)}\Big)^{2}}{\Big[\Big(\frac{k \pi}{2 \rho_*(t)}\Big)^{2}+j+1\Big]^{2}\Big[\Big(\frac{k \pi}{2 \rho_*(t)}\Big)^{2}+j\Big]^{2}}\\\nonumber
&\mm\mm\cdot\bigg\{2-\frac{2 \Big[2\Big(\frac{k \pi}{2 \rho_*(t)}\Big)^{2}+2 j+1\Big]}{\Big[\Big(\frac{k \pi}{2 \rho_*(t)}\Big)^{2}+j+1\Big]}
-\frac{2 \Big[2\Big(\frac{k \pi}{2 \rho_*(t)}\Big)^{2}+2j+1\Big]}{\Big[\Big(\frac{k \pi}{2 \rho_*(t)}\Big)^{2}+j\Big]}\bigg\}d t,
   \\\nonumber
 \frac{\partial^{2} k_2}{\partial j^{2}}(\rho_*)
 &=\int_{0}^{T}  \frac{1}{\rho_*(t)} \sum_{\substack{k=1\\  k \ odd}}^\infty\bigg\{\frac{2}{\Big(\frac{k \pi}{2 \rho_*(t)}\Big)^{2}+j}-\frac{4 j}{\Big[\Big(\frac{k \pi}{2 \rho_*(t)}\Big)^{2}+j\Big]^{2}}+\frac{2 j^{2}}{\Big[\Big(\frac{k \pi}{2 \rho_*(t)}\Big)^{2}+j\Big]^{3}}\bigg\}d t.
      \end{align}
From
\bea\nonumber
    \begin{split}
   & 2-\frac{2 \Big[2\Big(\frac{k \pi}{2 \rho_*(t)}\Big)^{2}+2 j+1\Big]}{\Big[\Big(\frac{k \pi}{2 \rho_*(t)}\Big)^{2}+j+1\Big]}
-\frac{2 \Big[2\Big(\frac{k \pi}{2 \rho_*(t)}\Big)^{2}+2j+1\Big]}{\Big[\Big(\frac{k \pi}{2 \rho_*(t)}\Big)^{2}+j\Big]}<0,\\
&\frac{2}{\Big(\frac{k \pi}{2 \rho_*(t)}\Big)^{2}+j}-\frac{4 j}{\Big[\Big(\frac{k \pi}{2 \rho_*(t)}\Big)^{2}+j\Big]^{2}}+\frac{2 j^{2}}{\Big[\Big(\frac{k \pi}{2 \rho_*(t)}\Big)^{2}+j\Big]^{3}}=\frac{2 \Big(\frac{k \pi}{2 \rho_*(t)}\Big)^{4}}{\Big[\Big(\frac{k \pi}{2 \rho_*(t)}\Big)^{2}+j\Big]^{3}}>0,
\end{split}
\eea
we get $P_{j}(\rho_*)<0$ for $j>j_0$. The \re{eded} implies the result.
\end{proof}

Lemmas \ref{lem6.2} and \ref{lem5.3} imply that $\gamma_j$ is either strictly decreasing in $j$ or first increasing and then decreasing.
We study the positive integer $j(>j_0)$ with $j=n^2+m^2$.
If $\gamma_{j}$ is distinct from other $\gamma_{i}(i\neq j)$,
then
\begin{align}\label{ker1}
    \mathrm{Ker}[G_{\tilde{\rho}}(0,\gamma_{j})]= \bigoplus\limits_{\{(n, m)\, : \,j=n^2+m^2\}} \spn\{S_{j,\gamma_j}(t)\cos(nx_1)\cos(mx_2)\}.
\end{align}
Otherwise, we get that  $\gamma_{j_1}=\gamma_{j_2}(j_1<j_2)$ and
\begin{align}\label{ker}
    &\mathrm{Ker}[G_{\tilde{\rho}}(0,\gamma_{j_1})]\\\nonumber
    &=
\dis\bigoplus_{\substack{\{(n_{1}, m_{1})\,:\,j_1=n_{1}^2+m_{1}^2\}\\\{(n_{2}, m_{2})\,: \, j_2=n_{2}^2+m_{2}^2\}}}
\spn\{ S_{j_{1},\gamma_{j_1}}(t)\cos(n_{1}x_1)\cos(m_{1}x_2), S_{j_{2}, \gamma_{j_1}}(t)\cos(n_{2} x_1)\cos(m_{2} x_2) \}.
\end{align}

\eqref{ker1} implies that the dimension of $\mathrm{Ker}[G_{\tilde{\rho}}(0,\gamma_{j})]$ is determined
by the number of decompositions of $j$.
Note that if $(n,m)$ is a sum-of-square decomposition of $j$, then $(m,n)$ is also one of $j$. Also, $j$ may have
many pairs of sum-of-squares decompositions.
Hence $G_{\tilde{\rho}}(0,\gamma_{j})$
may have a high-dimensional kernel space.
From \re{ker},  we notice that $\mathrm{Ker}[G_{\tilde{\rho}}(0,\gamma_{j})]$ are more complicated for the case $\gamma_{j_1}=\gamma_{j_2}(j_1<j_2)$.
In order to reduce the dimension of the kernel space of $G_{\tilde{\rho}}(0,\gamma_{j})$, we restrict $G(\cdot,\gamma_j)$ on some subspaces with some periodicity and symmetry.

We divide the proof of Theorem \ref{result1} into two cases.

\begin{thm}\label{bif0} Suppose that $ \Phi\in C^{1+\frac{\alpha}{3}}{([0,\infty))}$ is a positive $T-$periodic function.
Assume that $\gamma_j$ is distinct from the other $\gamma_{i}(i\neq j)$ and $j=n^2+m^2(j>j_0)$.
Then the point $\gamma_{j}$
is a bifurcation value of problem \eqref{1.10a}--\eqref{1.17a}.
 More precisely, there exist at least $\alpha_j$ (given in Theorem \ref{result1}) bifurcation branches of non-flat
solutions bifurcating from $\gamma=\gamma_{j}$  with the free boundary
\begin{align}\label{b1}
y=\rho_* (t)+ \epsilon S_{j,\gamma_j}(t)\cos(nx_1)\cos(mx_2) + O(\epsilon^2).
\end{align}
If $j=n^2$, there is an additional branch with
the free boundary
\begin{align}\label{b2}
y=\rho_*(t) + \epsilon S_{j, \gamma_j}(t)[\cos(nx_1)+\cos(nx_2)] + O(\epsilon^2).
\end{align}
\end{thm}
\begin{proof}
At first, we show
\begin{align}\label{kongjianfenjie}
  W = \mathrm{Ker}[G_{\tilde{\rho}}(0,\gamma_{j})] \bigoplus \im[G_{\tilde{\rho}}(0,\gamma_{j})] .
\end{align}

In Lemma \ref{lem2.8}, we get that
$ G_{\tilde{\rho}}(0,\gamma) :Y  \cap  E  \rightarrow W  $ is a Fredholm operator with index zero.
All that's left for us to do is to prove
\begin{align} \nonumber
\mathrm{Ker}[G_{\tilde{\rho}}(0,\gamma_{j})]  \cap  \im[G_{\tilde{\rho}}(0,\gamma_{j})]=\{0\} .
  \end{align}
Let $u\in \mathrm{Ker}[G_{\tilde{\rho}}(0,\gamma_{j})]  \cap  \im[G_{\tilde{\rho}}(0,\gamma_{j})]$. Since $u\in \mathrm{Ker}[G_{\tilde{\rho}}(0,\gamma_{j})] $, we have $$u= \sum_{j=n^2+m^2} a_{n,m}S_{j,\gamma_{j}}\cos(nx_1)\cos(mx_2) .$$
Applying $u\in  \im[G_{\tilde{\rho}}(0,\gamma_{j})]$, we derive the existence of $v\in  Y \cap E  $ such that
\begin{align}\label{hh}
  G_{\tilde{\rho}}(0,\gamma_{j})v=\sum_{j=n^2+m^2} a_{n,m}S_{j,\gamma_{j}}\cos(nx_1)\cos(mx_2).
\end{align}
From \re{fff},
we obtain that $v$ has the following expression
\begin{align}\nonumber
v=\sum_{j=n^2+m^2} g_{n, m}(t) \cos (n x_1) \cos (m x_2).
\end{align}
Using \re{fff} and \re{hh}, we obtain
\begin{align}\nonumber
G_{\tilde{\rho}}(0,\gamma_{j})v
&=\sum_{j=n^2+m^2}\left[   g_{n, m}^{\prime}(t)+A_{j, \gamma_{j}}(t) g_{n, m}(t)            \right]\cos (n x_1) \cos (m x_2)\\\nonumber
&=\sum_{j=n^2+m^2} a_{n,m}S_{j,\gamma_{j}}\cos(nx_1)\cos(mx_2),
\end{align}
which implies
\begin{align}\nonumber
  g_{n, m}^{\prime}(t)+A_{j, \gamma_{j}}(t) g_{n, m}(t) =a_{n,m}S_{j,\gamma_{j}}(t),
\end{align}
that is,
 \begin{align}\nonumber
  g_{n, m}(t)=e^{-\int_{0}^{t} A_{j,\gamma_{j}}(s)ds}g_{n, m}(0)  +\int_{0}^{t} a_{n,m}S_{j,\gamma_{j}}(s) e^{-\int_{s}^{t} A_{j,\gamma_{j}}(\tau)d\tau}ds.
\end{align}
Since $g_{n, m}(T)=g_{n, m}(0)$, it follows that
$\int_0^T a_{n,m}S_{j,\gamma_{j}}(s) e^{-\int_s^T A_{j, \gamma_{j}}(\tau) d \tau} d s=0.$   The fact that $\int_{0}^{T}A_{j, \gamma_j}(s)ds\\=0$ gives $\int_0^T S_{j,\gamma_{j}}(s) e^{-\int_s^T A_{j, \gamma_{j}}(\tau) d \tau} d s=T$.
we get $a_{n,m}=0$.
Then $u=0$ and  \eqref{kongjianfenjie}  holds.

Therefore, \eqref{kongjianfenjie} implies
\begin{align}\nonumber
  Y \cap E  = \mathrm{Ker}[G_{\tilde{\rho}}(0,\gamma_{j})] \bigoplus \left(\im[G_{\tilde{\rho}}(0,\gamma_{j})]  \cap Y \cap E  \right).
\end{align}

If $j=n^2+m^2=n_1^2+m_1^2$ and $n\neq n_1(n, n_1\geq 1,  m, m_1\geq 0)$, then $S_{j, \gamma_j}(t)\cos(n_1x_1)\cos(m_1x_2)$ $\not\in\mathrm{Ker}[G_{\tilde{\rho}}(0,\gamma_{j})] \Big|_{Y \cap E _{n,m}} $.
Let $G^{n,m}(\tilde{\rho},\gamma_{j})=G(\tilde{\rho},\gamma_{j}) \Big|_{(Y \cap E )_{n,m}} $.
Lemma \ref{5.7} implies that $G^{n,m}(\cdot,\gamma)$ maps $(Y \cap E )_{n,m}$ to $W_{n,m} $.
Hence,
we get
$$\mathrm{Ker}[G^{n,m}_{\tilde{\rho}}(0,\gamma_{j})] =  \spn\{S_{j, \gamma_j}(t)\cos(nx_1)\cos(mx_2)\},$$
and
$$\dim(\mathrm{Ker}[G^{n,m}_{\tilde{\rho}}(0,\gamma_{j})] )=1.$$
Then $ \spn\{S_{j, \gamma_j}(t)\cos(nx_1)\cos(mx_2)\}  \bigoplus   \im[G^{n,m}_{\tilde{\rho}}(0,\gamma_{j})] = W_{n,m}$, i.e., $\codim\im[G^{n,m}_{\tilde{\rho}}(0,\gamma_{j})]  = 1$.

Differentiating \eqref{fff} in $\gamma$, we have
\bea\nonumber
&G^{n,m}_{\tilde{\rho}\gamma }(0,\gamma_{j})S_{j, \gamma_j}(t)\cos(nx_1)\cos(mx_2)= \dis\frac{ j^{\frac{3}{2}}}{2}\tanh (\sqrt{j}\rho_*(t)) S_{j, \gamma_j}(t)\cos(nx_1)\cos(mx_2),
\eea
and
\bea\nonumber
&G^{n,m}_{\tilde{\rho}\gamma }(0,\gamma_{j})S_{j, \gamma_j}(t)\cos(nx_1)\cos(mx_2) \notin \im[G^{n,m}_{\tilde{\rho}}(0,\gamma_{j})].
\eea

Consequently, conditions (i)¨C(iv) of Theorem \ref{thmcr} are true and we get that \eqref{b1} holds.

If $j=n^2$, we define $G^{n,+}(\tilde{\rho},\gamma_{j})=G(\tilde{\rho},\gamma_{j}) \Big|_{(Y \cap E)_{n,+}} $ and get
$$\mathrm{Ker}[G^{n,+}_{\tilde{\rho}}(0,\gamma_{j})] =  \spn\{S_{j, \gamma_j}(t)(\cos(nx_1)+\cos(nx_2))\}.$$
Similarly,
we obtain that \eqref{b2} is true.
\end{proof}

\begin{thm}\label{thm6.2}
Suppose that  $ \Phi\in C^{1+\frac{\alpha}{3}}{([0,\infty))}$ is a positive $T-$periodic function.
Assume that there exists another $i$ such that $\gamma_{i}=\gamma_{j}$ and  $\max\{i,j\}=n^2+m^2$ .
Then the point $\gamma_{j}$
is a bifurcation value of problem \eqref{1.10a}--\eqref{1.17a}. More precisely, there exist at least $\beta$ (given in Theorem \ref{result1}) bifurcation branches of non-flat
solutions bifurcating from $\gamma=\gamma_{j}$  with free boundary
$$y=\rho_*(t) + \epsilon S_{n^2+m^2, \gamma_{j}}(t)\cos(nx_1)\cos(mx_2) + O(\epsilon^2).$$

If $\min\{i,j\}=n^2$ or $\max\{i,j\}=n^2$, there is an additional branch with
the free boundary
\begin{align}\label{aaaa}
y=\rho_*(t) + \epsilon S_{n^2, \gamma_{n^2}}(t)[\cos(nx_1)+\cos(nx_2)] + O(\epsilon^2),
\end{align}
where $n=\sqrt{\max\{i,j\}}$ if both $\min\{i,j\}$ and $\max\{i,j\}$ are perfect square numbers.
\end{thm}

\begin{proof}
Set $\{ i\,|\,   \gamma_{i}=\gamma_{j}  \}=\{ j_1,j_2  \}(j_1<j_2)$.
We divide the argument into two cases.

Case A:
$j_2\neq k^2j_1$ for all positive integer $k$.

The fact that $j_2\neq k^2j_1$ for all positive integer $k$ implies that $(n_{2}, m_{2}) \neq k(n_{1}, m_{1})$ for all positive integer $k$.
We get that $S_{j_2, \gamma_{j_2}}(t)\cos(n_2x_1)\cos(m_2 x_2)\not\in (Y \cap E)_{n_1,m_1}$ and $S_{j_1, \gamma_{j_1}}(t)\cos(n_1x_1)\cos(m_1 x_2)\\\not\in (Y \cap E)_{n_2,m_2}$.
Similarly to the proof of Theorem \ref{bif0},
we obtain that $(0,\gamma_{j})$
is a bifurcation point, the number of branches at $(0,\gamma_{j})$ is $\alpha_{j_1}+\alpha_{j_2}$,
and the free boundaries are given by
\begin{align*}
y=\rho_*+ \epsilon S_{j_1, \gamma_{j_1}}(t) \cos(n_1 x_1)\cos(m_1 x_2) + O(\epsilon^2), \quad y=\rho_*+ \epsilon S_{j_2, \gamma_{j_2}}(t)\cos(n_2 x_1)\cos(m_2 x_2) + O(\epsilon^2).
\end{align*}
If $j_1=n^2$ or $j_2=n^2$, {there is  an   additional branch with
the free boundary}  \re{aaaa}.

Case B:
$j_2= k^2j_1$ for some positive integer $k$.

If $j_1=n_1^2+m_1^2$, then $k(n_{1}, m_{1})\in  \{(n_{2},m_{2})\,:\, j_2=n_{2}^2+m_{2}^2, n_{2}\geq m_{2}\}$.

Using Crandall-Rabinowitz theorem  to $ G^{n_2,m_2}$ where $j_2=n_2^2+m_2^2$,
 we obtain that $(0,\gamma_{j})$
is a bifurcation point and the number of branches at $(0,\gamma_{j})$ is $\alpha_{j_2}$.
Furthermore, the free boundary is given by
$$y=\rho_* + \epsilon S_{j_2, \gamma_{j_2}}(t)  \cos(n_2 x_1)\cos(m_2 x_2) + O(\epsilon^2).$$

If $j_1=n^2$, then  $j_2=k^2n^2$, there is an additional branch with
the free boundary
\begin{align*}
y=\rho_* + \epsilon S_{j_2, \gamma_{j_2}}(t) [\cos(knx_1)+\cos(knx_2)] + O(\epsilon^2).
\end{align*}
We complete the proof.
\end{proof}

\begin{rem}
    In previous steady-state bifurcation stemming from steady-state solutions of free boundary problems,
  $\mu$ or $\gamma$ has been taken as the bifurcation parameter for the branching problem and the complexity of the proof is the same.  Also, $\mu_j$ and $\gamma_j$ have a reciprocal relationship.
The reason that we choose $\gamma$ as a bifurcation parameter and not $\mu$ is that the complexity of the verification of the transversality condition is different.
From \eqref{1.15}, $\rho_*$ is dependent on $\mu$, but not on $\gamma$. This results in $\mu$ and $\gamma$ no longer being on the same footing.
Because $G_{\tilde{\rho}\mu }(0,\mu_{j})$ involves an unknown derivative $\frac{\partial \rho_* }{\partial \mu}$, the calculation of $G_{\tilde{\rho}\mu }(0,\mu_{j})$ and the verification of the transversality condition are much more complicated than that of $G_{\tilde{\rho}\gamma }(0,\gamma_{j})$.
\end{rem}

\appendix

\section{}

\subsection{Existence and Uniqueness of the Positive Global Solution and the Necessary and Sufficient Condition for Global Stability of Zero Equilibrium Solution
}
In this subsection, we shall give the existence and uniqueness of the positive global solution of \re{1.15}--\re{1.155} and the necessary and sufficient condition for global stability of zero equilibrium solution.
\vskip 3mm\noindent
{\it\bf Proof of Theorem \ref{thm:1.1a'}.}
 From the ODE theory, the local existence and uniqueness of the solution of \re{1.15}--\re{1.155} are obvious.

 Denote $ \Phi^{*}=\max _{0 \le t \le T} \Phi(t), \quad \Phi_{*}=\min _{0 \le t \le T} \Phi(t).$
 Since $\frac{ \tanh \rho}{\rho}$ is strictly decreasing in $\rho$ and $\frac{ \tanh \rho}{\rho}\in(0,1)$, we have
$$
-\tilde{\sigma}\rho(t) \leq \frac{d \rho}{d t}=\rho(t)\left[\Phi(t)  \frac{ \tanh( \rho(t))}{\rho(t)}-{\tilde{\sigma}}\right] \leq ( \Phi^{*}-\tilde{\sigma})\rho(t),
$$
which implies
\begin{align}\nonumber
 \rho(0) e^{-\tilde{\sigma} t} \leq \rho(t) \leq \rho(0) e^{( \Phi^{*}-\tilde{\sigma})t}.
\end{align}
 Hence, the solution doesn't blow up or disappear at a finite time.
 Then we get the results.
\hfill$\square$

\vskip 3mm

\noindent
{\it\bf Proof of Theorem \ref{thm:1.1'}.}
At first,  we prove the solution $\rho(t)\equiv0$ of \re{1.15} is globally stable if  $\overline{\Phi}<\tilde{\sigma}$.

 For any $t \in[0, T]$, from \re{1.15}, we have
$$
\rho(t+n T)=\rho(0) e^{\mu\displaystyle\int_{0}^{t+n T} \Phi(s) \frac{ \tanh( \rho(s))}{\rho(s)} d s-\mu\tilde{\sigma} (t+n T)}.
$$
Since $\frac{ \tanh \rho}{\rho}$ is monotone decreasing in $\rho>0$ and $\frac{ \tanh \rho}{\rho}<1$, for $\overline{\Phi}<\tilde{\sigma}$, we obtain
$$
\rho(t+n T) \le \rho(0) e^{\mu\displaystyle\int_{0}^{t+n T} \Phi(s)  d s-\mu\tilde{\sigma} n T} \le \rho(0)e^{\mu  T\overline{\Phi}}e^{\mu n T(\overline{\Phi}-\tilde{\sigma})} \rightarrow 0, \mm n \rightarrow \infty.
$$

 Next, we prove the zero equilibrium solution  of \re{1.15} is globally stable if  $\widetilde{\sigma}=\overline{\Phi}$. We shall show
\begin{align}
&\rho(t+T)\le \rho(t) \qquad for ~t>0, \label{3.33}\\
&\rho(t)\le \rho(a)e^{ \displaystyle\mu(\Phi^{*}-\widetilde{\sigma})T}  \qquad for ~  t\in[a,a+T], \qquad a\ge 0, \ \label{3.2}\\
&\liminf _{t \rightarrow+\infty}  \rho(t)=0, \ \label{3.9}\\
&\lim _{t \rightarrow+\infty} \rho(t)=0. \ \label{3.16}
\end{align}

Indeed, since $0<\frac{ \tanh \rho}{\rho}<1$, \re{1.15} implies
\begin{align}
 \frac{d \rho}{d t}\le \mu \rho(t)\left[\Phi(t)-\widetilde{\sigma}\right]. \ \label{3.222}
\end{align}
Then
$
\rho(t+T) \leq \rho(t)e^{ \int_{t}^{t+T}\mu\left[\Phi(t)-\widetilde{\sigma}\right] d t}=\rho(t)e^{  \displaystyle\mu\left[ \overline{\Phi}T  -\widetilde{\sigma}T\right]}=\rho(t),
$
i.e., \re{3.33} is true.

Applying \re{3.222}, we obtain
$
 \frac{d \rho}{d t}\le \mu \rho(t)\left[\Phi^{\ast}-\widetilde{\sigma}\right],
$
which implies
$$
 \rho(t)\le \rho(a)e^{ \displaystyle \mu(\Phi^{*}-\widetilde{\sigma} )(t-a)}\le \rho(a)e^{ \displaystyle \mu(\Phi^{*}-\widetilde{\sigma} )T},
$$
for  $t\in[a,a+T]$. Then \re{3.2} holds.

 We use the contradiction method to prove  \re{3.9}. Assume  that
$
\liminf _{t \rightarrow+\infty}  \rho(t)=\alpha>0.
$
For $\forall$ $\varepsilon\in(0, \alpha)>0$, there exists $M>0$ such that
\begin{align}
\rho(t)>\alpha-\varepsilon  \qquad for ~    t>M. \ \label{3.188}
\end{align}
Using \re{1.15} and the monotonicity of  $\frac{ \tanh \rho}{\rho}$ in $\rho$, we get
\begin{align}\nonumber
 \frac{d \rho}{d t}
 \le \mu \rho(t)\left[\Phi(t) \frac{ \tanh( \alpha-\varepsilon)}{\alpha-\varepsilon} -\widetilde{\sigma}\right] \qquad for ~   t>M.
\end{align}
Then
\begin{equation}\label{3.13}
\rho(t^{\ast}+nT) \le \rho(t^{\ast})e^{ \displaystyle \int_{t^{\ast}}^{t^{\ast}+nT}\mu\left[\Phi(t) \frac{ \tanh( \alpha-\varepsilon)}{\alpha-\varepsilon} -\widetilde{\sigma}\right] d t}=\rho(t^{\ast})e^{ \displaystyle\mu nT \left[\frac{ \tanh( \alpha-\varepsilon)}{\alpha-\varepsilon}\overline{\Phi}-\widetilde{\sigma}\right]},
\end{equation}
for   $t^{\ast}>M$ and  $n\geq1$ is an integer.
Applying \eqref{3.13} and the fact
$
\frac{ \tanh( \alpha-\varepsilon)}{\alpha-\varepsilon}\overline{\Phi}-\widetilde{\sigma}<\overline{\Phi}-\widetilde{\sigma}=0,
$
we get
\begin{align}\nonumber
 \rho(t^{\ast}+nT) \rightarrow0, \qquad n\rightarrow\infty.
\end{align}
It contracts with \re{3.188}, i.e., \re{3.9} is true.

Now we turn to \re{3.16}.
Using \re{3.9}, for $\forall$ $\varepsilon>0$, there exist $M_{0}>0$ and a sequence $t_{n}\rightarrow\infty$  such that
$
\rho(t_{n})<\varepsilon$  for     $t_{n}>M_{0}$.
\re{3.33} implies
$
\rho(t_{N}+kT)\le \rho(t_{N})< \varepsilon$
for $t_{N}>M_{0}$ and the integer $k\geq1$. For $t>t_{N}$, there exists $k_{0}$ such that $t\in[ t_{N}+k_{0}T,t_{N}+(k_{0}+1)T)$. Applying \re{3.2}, we have
\begin{align}\begin{array}{ll}\nonumber
\rho(t) \le \rho(t_{N}+k_{0}T)e^{ \displaystyle\mu(\Phi^{*}-\widetilde{\sigma})T} \le \varepsilon  \displaystyle e^{  \displaystyle\mu (\Phi^{*}-\widetilde{\sigma}) T}.
\end{array}
\end{align}
Then \re{3.16} holds.

As so far, we have shown that if $\widetilde{\sigma}\ge\overline{\Phi}$, $\rho(t)\equiv0$  is globally stable.
Finally, we prove that if $\rho(t)\equiv0$ of \re{1.15}--\re{1.155} is globally stable, then $\widetilde{\sigma}\ge\overline{\Phi}$.

Since  $\lim _{t \rightarrow \infty} \rho(t)=0$, for $\forall$ $\varepsilon>0$, there exists $M_{1}>0$ such that $\rho(t)<\varepsilon$ for $t \ge M_{1}$. Then the fact that $\frac{ \tanh \rho}{\rho}$ is strictly decreasing implies

\begin{align}\nonumber
\frac{d \rho}{\rho}
\ge \mu\left[\Phi(t)  \frac{ \tanh( \varepsilon)}{\varepsilon}-\tilde{\sigma}\right] dt, \quad t \ge M_{1}.
\end{align}
Hence
$
\frac{\rho(t+T)}{\rho(t)} \ge e^{\mu T\left(\overline{\Phi}  \frac{ \tanh( \varepsilon)}{\varepsilon}-\tilde{\sigma}\right)}.
$
 If $\overline{\Phi}>\tilde{\sigma}$, we choose $\varepsilon$  such that $ \frac{ \tanh( \varepsilon)}{\varepsilon}>\frac{\tilde{\sigma}}{ \overline{\Phi}}$.  Then
$
\frac{\rho(t+T)}{\rho(t)} >1,
$
 which contradicts to the $\lim _{t \rightarrow \infty} \rho(t)=0$. Thus $\widetilde{\sigma}\ge\overline{\Phi}$ holds.
\hfill$\square$

\subsection{Existence, Uniqueness and Stability of the Periodic Solution
}\label{result}

In this subsection, we shall prove the existence, uniqueness and stability of the periodic solution of \re{1.15}--\re{1.155}.

\vskip 3mm\noindent
{\it\bf Proof of Theorem \ref{thm:1.2}.}
Because $\widetilde{\sigma}<\overline{\Phi}\le\Phi^{\ast}$ , $0<\frac{ \tanh \rho}{\rho}<1$ and $\frac{ \tanh \rho}{\rho}$ is strictly decreasing, we have $x_{2}=(\frac{ \tanh \rho}{\rho})^{-1}(\frac{\widetilde{\sigma}}{\Phi^{\ast}})$ and $\overline{x}=\frac{(\frac{ \tanh \rho}{\rho})^{-1}(\frac{\widetilde{\sigma}}{\overline{\Phi}})}{e^{\mu(\Phi^{*}-\widetilde{\sigma})T}}$ are well defined, and
$
\overline{x}<x_{2}.
$
For each $\rho_{0}\in[\overline{x}, x_{2}]$, we define the mapping $F$: $[\overline{x}, x_{2}]\rightarrow \mathbb{R}$ by
 $F(\rho_{0})=\rho(T)$, where $\rho$ is the solution of \re{1.15}--\re{1.155}.

At first, we prove that $F$ maps $[\overline{x}, x_{2}]$ into $[\overline{x}, x_{2}]$. Let $\rho(0)\in[\overline{x}, x_{2}]$.
Notice that $x_{2}$ is the upper solution of \re{1.15}--\re{1.155}. Applying the comparison theorem, we have
$
\rho(t)\le x_{2} $  for $t>0.
$
Hence
\begin{align}
\rho(T)\le x_{2}. \ \label{4.3}
\end{align}
We define $\overline{\rho}$ by the solution of \re{1.15} with $\overline{\rho}(0)=\overline{x}$.
The comparison theorem implies
\begin{align}
\rho(t)\ge \overline{\rho}(t) \qquad for ~ t>0. \ \label{4.6}
\end{align}
Applying  $0<\frac{ \tanh \rho}{\rho}<1$,  we have
$
  \frac{d \overline{\rho}}{d t}
  \le\mu \overline{\rho}(t)\left[\Phi^{\ast}- \widetilde{\sigma}\right].
$
Then
\begin{equation}\nonumber
\overline{\rho}(t) \leq \overline{x} e^{ \displaystyle\mu(\Phi^{*}-\widetilde{\sigma}) t } \leq \overline{x} e^{ \displaystyle\mu(\Phi^{*}-\widetilde{\sigma}) T}=\left(\frac{ \tanh \rho}{\rho}\right)^{-1}\left(\frac{\widetilde{\sigma}}{ \overline{\Phi}}\right) \qquad for ~t\in[0,T].
\end{equation}
Applying the fact that  $\frac{ \tanh \rho}{\rho}$ is strictly decreasing, we get
$
\frac{ \tanh \overline{\rho}(t)}{\overline{\rho}(t)}\geq\frac{\widetilde{\sigma}}{\overline{\Phi}}$ for $t\in[0, T].
$
Hence,
\begin{align*}
  \frac{d \overline{\rho}}{d t}\ge  \mu \overline{\rho}(t)\left[ \Phi(t)\frac{\widetilde{\sigma}}{\overline{\Phi}}- \widetilde{\sigma}\right] \qquad for ~ t\in(0, T),
\end{align*}
which implies
\begin{align}
\overline{\rho}(T)\ge \overline{\rho}(0)e^{\displaystyle \int_{0}^{T}\mu\left[\Phi(t)\frac{\widetilde{\sigma}}{\overline{\Phi}}-\widetilde{\sigma}\right]dt}=\overline{\rho}(0)=\overline{x}.
\label{4.10}
\end{align}
From \re{4.3}, \re{4.6} and \re{4.10}, we have
$\rho(T)\in[\overline{x}, x_{2}].$
Hence  $F$ maps $[\overline{x}, x_{2}]$ into $[\overline{x}, x_{2}]$. From the continuous dependence of the solution $\rho$ on the initial value $\rho_{0}$, we have  $F$ is continuous. Using Brouwer's fixed point theorem,  it follows  that $F$ has a fixed point $\rho_{\ast}(0)$. Then the solution $\rho_{\ast}$ of \re{1.15}--\re{1.155}  with $\rho(0)=\rho_{\ast}(0) $ is a $T-$periodic  positive solution. We shall prove the uniqueness of the periodic solution later.

Next we turn to  prove $(ii)$.
Let
\begin{align}\nonumber
\rho_{min}=\min\limits_{t>0}\left\{\rho_{\ast}(t)\right\}\qquad \text{and} \qquad \rho_{max}=\max\limits_{t>0}\left\{\rho_{\ast}(t)\right\}.
\end{align}
The uniqueness of the solution of \re{1.15}--\re{1.155} implies that $\rho_{min}>0$ and $\rho_{max}>0$.
Assume  $\rho(t)$ is the solution of \re{1.15}--\re{1.155}  with  $\rho(0)>0$. Let
\begin{align}\label{qq}
\rho(t)=\rho_{\ast}(t) e^{y(t)}.
\end{align}
In order to prove $
|\rho(t)-\rho_{\ast}(t)|\le C e^{-\delta t}$,
 we need to  prove that there exist $\delta>0$ and $C>0$ such that
\begin{align}\nonumber
|e^{y(t)}-1|\le C e^{-\delta t}\qquad for ~    t>0.
\end{align}
Taking \re{qq} into \re{1.15}--\re{1.155}, we get
\begin{align}\begin{array}{ll}
y'(t) =  \mu\Phi(t)\Big[ \displaystyle\frac{ \tanh (\rho_{\ast}(t)e^{y(t)})}{\rho_{\ast}(t)e^{y(t)}}  -\displaystyle\frac{ \tanh (\rho_{\ast}(t))}{\rho_{\ast}(t)} \Big].
\end{array}
\label{4.12}
\end{align}
The uniqueness of the solution of \re{1.15}--\re{1.155} implies that $\rho(t)>\rho_{*}(t)$  if $\rho(0)>\rho _{*}(0)$ and $\rho(t)<\rho_{*}(t)$ if $\rho(0)<\rho _{*}(0)$. Then $y(t)>0$ if  $y(0)>0$ and $y(t)<0$ if  $y(t)<0$. Therefore, according to the sign of $y(t)$, we divide the arguments into two cases.

Case 1: $y(t)>0$.

Applying \re{4.12} and the strictly monotonicity  of  $\frac{ \tanh \rho}{\rho}$ in $\rho$, we have $y'(t)<0$.
Using  \re{4.12}  and the mean value theorem, we have
\begin{align}\nonumber
y'(t) e^{y(t)}=  \mu\Phi(t) (\displaystyle\frac{ \tanh \rho}{\rho})'(\zeta(t))\rho_{\ast}(t)(e^{y(t)}-1)e^{y(t)}   \le -\mu\Phi_{\ast}M_{min} \rho_{min}   (e^{y(t)}-1)\nonumber,
\end{align}
here,  $\zeta(t)\in[\rho_{\ast}(t), \rho_{\ast}(t)e^{y(t)}]\subseteq [\rho_{min}, \rho_{max}e^{y(0)}]$
and $M_{min}=\min\limits_{x\in [\rho_{min}, \rho_{max}e^{y(0)}]}\left\{-(\frac{ \tanh \rho}{\rho})'\right\}>0$.
Therefore,
\begin{align}\nonumber
e^{y(t)}-1\le (e^{y(0)}-1)e^{-\mu\Phi_{\ast}M_{min} \rho_{min} t}\qquad for ~    t>0.
\end{align}

Case 2: $y(t)<0$.

The analysis of $y(t)<0$ is similar to Case 1  excepting that $M_{min}$ replaces $\overline{M}_{min}=\\\min  \limits_{ x\in [\rho_{min}e^{y(0)}, \rho_{max}]}\left\{(-\frac{ \tanh \rho}{\rho})'\right\}$ $>0$, we omit it.
Taking $\delta=\min\{\mu\Phi_{\ast} {M}_{min}\rho_{min} , \mu\Phi_{\ast}\overline{M}_{min} \rho_{min} e^{y(0)}\}$ and $C=|1-e^{y(0)}|$,   we get \re{1.17}.

Finally, we prove the  uniqueness of  solution $\rho_{\ast}(t)$. Otherwise, using \re{1.17}, we have
$$
\left|\rho_{*}^{1}(t)-\rho_{*}^{2}(t)\right|\leq\left|\rho(t)-\rho_{*}^{1}(t)\right|+\left|\rho(t)-\rho_{*}^{2}(t)\right| \rightarrow 0 \qquad t\rightarrow \infty,
$$
i.e., $\rho_{*}^{1}(t)=\rho_{*}^{2}(t)$. Then we get the result.\hfill$\square$

\section*{Acknowledgement}

{\small The first author was supported by the National Natural Science Foundation of China (12401261), the Fundamental Research Program of Shanxi Province, China
(202203021222011), and the Special Fund for Science and Technology Innovation Teams of Shanxi Province, China (202204051002015). The second author was supported by the National Natural Science Foundation of China (12171120).
The third author was supported by the Guangdong Basic and Applied Basic Research Foundation (2025A1515010922).}

\bigskip

\bibliographystyle{unsrt}
\bibliography{main}

\end{document}